\documentclass[reqno,12pt]{amsart}

\usepackage[active]{srcltx}
\usepackage{color}
\usepackage{amsmath,amsthm,amssymb}
\usepackage{epsfig}
\usepackage{graphicx}
\usepackage{amssymb}
\usepackage{latexsym}
\usepackage{pdfpages}

\usepackage{enumitem}

\usepackage{ulem}

\usepackage{ifpdf}

\ifpdf
\usepackage[hidelinks]{hyperref}
\else
\fi

\usepackage[usenames,dvipsnames]{pstricks}

\newtheorem{theorem}{Theorem}[section]
\newtheorem{lemma}[theorem]{Lemma}
\newtheorem{definition}[theorem]{Definition}
\newtheorem{proposition}[theorem]{Proposition}
\newtheorem{corollary}[theorem]{Corollary}
\newtheorem{remark}[theorem]{Remark}
\newtheorem{example}[theorem]{Example}

\def\R{\mathbb R}
\def\N{\mathbb N}

\def\1{\raisebox{2pt}{\rm{$\chi$}}}
\def\z{{\bf z}}

\numberwithin{equation}{section}

\def\Xint#1{\mathchoice
	{\XXint\displaystyle\textstyle{#1}}%
	{\XXint\textstyle\scriptstyle{#1}}%
	{\XXint\scriptstyle\scriptscriptstyle{#1}}%
	{\XXint\scriptscriptstyle\scriptscriptstyle{#1}}%
	\!\int}
\def\XXint#1#2#3{{\setbox0=\hbox{$#1{#2#3}{\int}$}
		\vcenter{\hbox{$#2#3$}}\kern-.5\wd0}}

\def\dashint{\Xint-}

\newcommand{\twopartdef}[4]
{
	\left\{
	\begin{array}{ll}
		#1 & #2 \\
		#3 & #4
	\end{array}
	\right.
}

\newcommand{\threepartdef}[6]
{
	\left\{
	\begin{array}{lll}
		#1 & \mbox{if } #2 \\
		#3 & \mbox{if } #4 \\
		#5 & \mbox{if } #6
	\end{array}
	\right.
}

\definecolor{violet(ryb)}{rgb}{0.53, 0.0, 0.69}
\begin{document}
	
\title[Gauss-Green Formula and Least Gradient Functions]{\bf The Anzellotti-Gauss-Green Formula and Least Gradient Functions in  Metric Measure Spaces}
	
\author[W. G\'{o}rny and J. M. Maz\'on]{Wojciech G\'{o}rny and Jos\'e M. Maz\'on}
	
\address{ W. G\'{o}rny: Faculty of Mathematics, Universit\"at Wien, Oskar-Morgerstern-Platz 1, 1090 Vienna, Austria; Faculty of Mathematics, Informatics and Mechanics, University of Warsaw, Banacha 2, 02-097 Warsaw, Poland;
\hfill\break\indent
{\tt  wojciech.gorny@univie.ac.at }
}

\address{J. M. Maz\'{o}n: Departamento de An\`{a}lisis Matem\`atico,
Universitat de Val\`encia, Dr. Moliner 50, 46100 Burjassot, Spain.
\hfill\break\indent
{\tt mazon@uv.es }}
	
	
\setcounter{tocdepth}{1}

\date{\today}
	
\begin{abstract} 
In the framework of the first-order differential structure introduced by Gigli, we obtain a Gauss-Green formula on regular bounded open sets of metric measure spaces, valid for BV functions and vector fields with integrable divergence. Then, we study least gradient functions in metric measure spaces using this formula as the main tool.
\end{abstract}
	
\maketitle

{\renewcommand\contentsname{Contents}
\setcounter{tocdepth}{3}
\tableofcontents}

\section{Introduction}

This paper has two principal goals. The first one is to obtain a Gauss-Green formula of Anzellotti type in bounded open sets of metric measure spaces, and the second one is to apply said formula to study the least gradient problem. The Gauss-Green formula is one of the most important tools in pure and applied mathematics. Motivated by its applications, throughout history there have been different generalisations,  which were of two main types: one may consider weakly differentiable vector fields but fairly regular (e.g. Lipschitz) domains, or, following De Giorgi and Federer, consider fairly regular vector fields and sets of finite perimeter (see e.g. \cite[Theorem 3.36]{AFP}).

The Gauss-Green formulas on regular domains were typically formulated for Lipschitz or Sobolev functions. However, for many further applications, it is not sufficient; for example, the characteristic function of a set $E$ of finite perimeter is not a Sobolev function, but in problems related to image processing, naturally you have to deal with characteristic functions of sets. Thus, a version of the Gauss-Green formula valid for functions of bounded variation (i.e. functions in $L^1$ whose weak derivatives are Radon measures) was needed. Motivated by problems in elasticity, Anzellotti in \cite{Anz} introduced several parings between bounded vector fields and the gradients of functions of bounded variation, and showed existence of traces of the normal component of bounded vector fields whose divergence is a Radon measure on the boundary of open bounded sets $\Omega \subset \R^N$ with Lipschitz boundary  (simultaneously, and for similar reasons, these objects have been studied by Kohn and Temam in \cite{KT}). In the literature, these objects are called in short divergence-measure vector fields and normal traces. In this context, Anzellotti proved a Gauss-Green formula,  which relates a divergence-measure vector fields, its normal trace, a BV function and its trace. In this form, the Gauss-Green formula has many applications. For instance, it is a fundamental tool for the study of the $1$-Laplacian operator $\Delta_1:= {\rm div}\left( \frac{Du}{\vert Du \vert}\right)$,  where the greatest difficulty is to define the quotient $\frac{Du}{\vert Du \vert}$, where $Du$ is a Radon measure which may vanish on a set of positive Lebesgue measure (this is in fact a typical property of solutions to some of the problems involving the $1$-Laplacian operator). This difficulty was overcome in \cite{ACMBook,ABCMNeu} using Anzellotti’s pairings; the role of this quotient is played by a vector field $\z$ such that $\Vert \z \Vert_\infty \leq 1$ and $(\z, Du) = \vert Du \vert$, where $(\z, Du)$ is the Anzellotti pairing.

After the work of Anzellotti in \cite{Anz}, the notion of divergence-measure fields was rediscovered (for different purposes) in the early $2000$s. In particular, Chen and Frid proved generalized Gauss-Green formulas for divergence-measure fields on open bounded set with Lipschitz deformable boundary (see \cite{CF1}), motivated by applications to the theory of systems of conservation laws.  The approach by Chen and Frid has led to a large number of extensions, see for example \cite{CCT,CT,CTZ}. Recently Anzellotti's pairing theory has been studied in more detail in \cite{CDeC1} and \cite{CDeC2}, with the main focus of describing the interaction between the discontinuity sets of a BV function and a divergence-measure vector field.

The Gauss-Green formula has been studied also in a number of different non-Euclidean contexts (see for example \cite{HN}, \cite{LY}).  Recently, various authors introduced a number of generalisations in the framework of metric measure spaces. In \cite{CM}, the authors established a Gauss-Green theorem for a set of finite perimeter and divergence-measure vector fields in Carnot groups. In the case more relevant to our considerations, when the metric space is doubling and satisfies a Poincar\'{e} inequality, a Gauss-Green formula on regular balls has been obtained in \cite{MMS}. Then, Buffa, Comi and Miranda Jr. in \cite{BCM} introduced the notion of regular domains, for which they proved existence of boundary traces of $L^\infty$ divergence-measure vector fields and a Gauss-Green formula valid for such vector fields and Lipschitz functions. In order to prove the Gauss-Green formula in this very general setting, the authors of \cite{BCM} extensively rely on the first-order differential structure introduced by Gigli (see \cite{Gig}); this construction produces a notion of cotangent and tangent modules, through which the vector fields with integrable divergence are defined.

Our first goal in this paper is to generalise this result and prove a general Gauss-Green formula on regular domains in metric measure spaces, which is valid for BV functions and vector fields with integrable divergence, similarly to the results proved by Anzellotti in \cite{Anz} in the Euclidean case. To this end, we first introduce the notion of Anzellotti pairings on metric measure spaces and study some of their properties. These are also expressed in terms of the first-order differential structure introduced by Gigli in \cite{Gig}.

Our second goal is to apply our Gauss-Green formula to study the least gradient problem in regular bounded open subsets of a metric measure space. In the last few years, this minimalisation problem
\begin{equation}\label{eq:originalproblem}
\min \bigg\{ \int_{\Omega} |Du|: \, u \in BV(\Omega), \, u|_{\partial\Omega} = f \bigg\}
\end{equation}
and its variants have attracted considerable attention, see for instance \cite{DS,GRS2017NA,JMN,MazRoSe,SWZ}. In the classical setting, $\Omega \subset \mathbb{R}^N$ is an open bounded set with Lipschitz boundary. This problem was first considered in this form in \cite{SWZ}, where the authors studied it from the point of view of geometric measure theory. Indeed, due to the results in \cite{BdGG}, the problem can be equivalently seen in the following way: to find a foliation of the domain by minimal surfaces (except for a set on which the solution is locally constant) in a way enforced by the boundary data. Under some geometric assumptions on the curvature of $\partial\Omega$, the authors established that for continuous boundary data there exists a unique solution, which is continuous up to the boundary, and studied its properties.

 When the boundary datum is not continuous, the situation is a bit different. Denote by $\mathbb{D}$ the unit disk in $\R^2$. In \cite{ST}, the authors give an example of a function  $f \in L^1(\partial \mathbb{D}) \backslash C(\partial\mathbb{D})$ such that the variational problem
$$\min \bigg\{ \int_{\mathbb{D}} |Du|: \, u \in BV(\mathbb{D}), \, u|_{\partial\mathbb{D}} = f \bigg\} $$
has no solution. Therefore, for $L^1$ data on the boundary, the correct formulation of the least gradient problem is not \eqref{eq:originalproblem}. To overcome this problem, a second point of view was introduced in \cite{MazRoSe}, where the authors studied the relaxed version of the problem
\begin{equation}\label{eq:relaxedproblem}
\min \bigg\{ \int_{\Omega} |Du| + \int_{\partial\Omega} |u - f| \, d\mathcal{H}^{N-1}: \, u \in BV(\Omega) \bigg\}
\end{equation}
and provide a characterisation of the solutions. This approach has the advantage that solutions to this problem exist for general boundary data $f \in L^1(\partial\Omega,\mathcal{H}^{N-1})$ and without any regularity assumptions on the domain. However, the boundary condition in \eqref{eq:relaxedproblem} is weaker than the boundary condition in \eqref{eq:originalproblem}, where it typically is understood for discontinuous boundary data as a trace of a BV function.

In \cite{JMN}, the authors study the following general least gradient problem
\begin{equation}\label{eq:anisoproblem}
\min \bigg\{ \int_{\Omega} \phi(x,Du): \, u \in BV(\Omega), \, u|_{\partial\Omega} = f \bigg\},
\end{equation}
where $\phi$ is a metric integrand and $f \in C(\partial \Omega)$. Under some restrictions on the metric integrand $\phi$ and assuming that $\Omega$ satisfies a  certain curvature condition, in \cite{JMN} it is proved that the problem \eqref{eq:anisoproblem} has a unique minimiser for every $f \in C(\partial \Omega)$. For the special case $\phi(x,\xi) := a(x)|\xi|$, the problem \eqref{eq:anisoproblem} is the weighted least gradient problem
\begin{equation}\label{eq:weitedblemrel}
\min \bigg\{ \int_{\Omega} a(x) \vert Du\vert  \, : \, u \in BV(\Omega), \ \  u|_{\partial\Omega} = f \bigg\},
\end{equation}
which appears in \cite{NTT} in connection with the conductivity imaging problem (see \cite{NTT1} for a good survey about this problem). Let us describe it shortly.

In \cite{HMN} the authors present a method for recovering the conformal factor of an anisotropic conductivity matrix in a known conformal class from one interior measurement. They assume that the matrix-valued conductivity function is of the form
$$\sigma = c(x) \sigma_0(x),$$
with $\sigma_0(x)$ known  and with the cross-property factor
$c(x)$ a scalar function to be determined. In \cite{HMN} the authors showed that the corresponding voltage potential $u$ is the unique solution of a general least gradient problem \eqref{eq:anisoproblem}, where $\phi$ is given by
$$\phi(x,\xi) = a(x) \left( \sum_{i,j=1}^N \sigma_0^{i,j}(x) \xi_i\xi_j \right)^{\frac12}, \quad a = \sqrt{\sigma_0^{-1} J \cdot J},$$
where $J = - \sigma \nabla u$ is the current density vector field generated by imposing the voltage $f$ at $\partial \Omega$. In particular, if $\sigma_0$ is the identity matrix, we recover the weighted least gradient problem \eqref{eq:weitedblemrel}. The corresponding Euler-Lagrange equation is
$${\rm div} \left( \sqrt{\vert g \vert} \frac{g^{i,j} \nabla_i u}{\vert \vert g^{-1} \nabla u \vert \vert_g} \right) = 0, \quad \hbox{in}  \ \Omega,$$
where $g_{i,j}$ is the Riemannian metric on $\Omega$:
$$g_{i,j}:= \left( \vert \sigma_0 \vert \vert J \vert_{\sigma_0^{-1}} \right)^{\frac{1}{N-1}} (\sigma_0^{-1})_{i,j}.$$
As a consequence of this, they showed the geometrical result that
equipotential sets $u^{-1}(\lambda):=  \{x \in \overline{\Omega} \ : \ u(x) = \lambda \}$ are in fact minimal surfaces with respect to the Riemannian metric $g_{i,j}$ on $\Omega$.  Finally, let us note that a relaxed version of problem \eqref{eq:anisoproblem} in the spirit of \eqref{eq:relaxedproblem} appears independently in connection with a problem in image processing (see \cite{CFM}) and is studied in detail in \cite{Mazon}.

In light of this, it is natural to consider the least gradient problem in a metric measure space different to the Euclidean one. In this paper, we aim to generalise the results of \cite{Mazon} and \cite{MazRoSe} to the setting of metric measure spaces and provide an Euler-Lagrange characterisation of solutions. This will be achieved using a linear first-order differential structure introduced recently by Gigli, see \cite{Gig}. Some variants of the least gradient problem in metric measure spaces have already been studied in \cite{HKLS,LMSS,KLLS}; similarly to the Euclidean case, there are several possible definitions, which we will now briefly review.

The simplest definition (see \cite{BdGG,Mir} in the Euclidean case) does not directly take into account the boundary condition. Given an open set $\Omega \subset \mathbb{X}$, we say that $u \in BV(\Omega, d, \nu)$ is a function of {\it least gradient} in $\Omega$, if for all $v \in BV(\Omega, d, \nu)$ with compact support
\begin{equation}\label{eq:leastgradient}
\vert Du \vert_\nu(\Omega)  \leq \vert D(u + v) \vert_\nu (\Omega).
\end{equation}
For the definitions of BV spaces in metric measure spaces, see Section \ref{sec:preliminaries}.

For Lipschitz boundary data, two possible notions of solutions have been introduced in \cite{KLLS}. Given an open bounded set $\Omega \subset \mathbb{X}$ and a boundedly supported function $f \in \hbox{Lip}(\mathbb{X})$, we say that a function $u$ is a {\it solution to the Dirichlet problem of least gradient} with boundary data $f$ in the sense of (B) (respectively in the sense of (T)), if it is a solution to the following minimisation problem:

\noindent (B) \ Minimise $\vert Dv \vert_\nu (\overline{\Omega})$ over all functions $v \in BV(\mathbb{X}, d, \nu)$ with $v = f$ on $\mathbb{X} \setminus \overline{\Omega}$.

\noindent (T) \ Minimise $\vert Dv \vert_\nu (\Omega) +\int_{\partial \Omega} \vert T_+ v - f \vert (x) dP_+(\Omega, x)$ over all functions $v \in BV(\mathbb{X}, d, \nu)$.

 Here, $P_+(\Omega,\cdot)$ is the inner perimeter measure, see Definition \ref{dfn:innerperimeter}. Both of these definitions are in the spirit of the relaxed formulation \eqref{eq:relaxedproblem}. The advantage of the first approach is that it is much easier to state the problem and prove existence of solutions using the direct method. On the other hand, in the second approach the solution depends only on the shape of $\Omega$ and the structure of $\mathbb{X}$ outside $\Omega$ is not taken into account. Existence of solutions in both cases was proved in \cite{KLLS} using approximations by $p$-harmonic functions; we also refer to that paper for a discussion regarding the relationship between the two formulations.

A stronger notion of solutions was introduced in \cite{LMSS}. In line with the classical result by Sternberg, Williams and Ziemer (see \cite{SWZ}), they assume more about the geometry of $\Omega$, but require that the trace of solution is equal to the boundary datum. Namely, assuming a condition which generalises positive mean curvature of $\partial\Omega$, they prove existence of a solution $u \in BV(\mathbb{X},d,\nu)$ in the sense of (B) for continuous boundary data $f \in BV(\mathbb{X},d,\nu) \cap C(\mathbb{X})$. Furthermore,
$$\lim_{\Omega \ni y \to x} u(y) = f(x)$$
whenever $x \in \partial \Omega$. In particular, $T_\Omega u = f|_{\partial\Omega}.$

In this paper, we will adopt the definition (T) of solutions to the least gradient problem and extend it to boundary data $f \in L^1(\partial\Omega,\mathcal{H})$.  Here, $\mathcal{H}$ is the codimension one Hausdorff measure defined via equation \eqref{eq:hausdorffcodimone}, and the exact definition of the functional we minimise will be given in equation \eqref{fuctme}. In \cite{KLLS}, as the authors note in the introduction, the proof of existence of solutions requires Lipschitz regularity of boundary data; this limitation follows from the lack of an appropriate Gauss-Green formula in the setting of metric measure spaces.  Since we provide a general enough Gauss-Green formula, we will be able to work directly with $L^1$ functions on the boundary.

The structure of the paper is as follows. In Section \ref{sec:preliminaries}, we recall the notions of Sobolev and BV functions in metric measure spaces and some of their properties. We also recall briefly the construction of the first-order differential structure on a metric measure space. Then, in Section \ref{sec:Anzellotti}, we prove the general Gauss-Green formula in the metric setting; to this end, we first introduce the notion of Anzellotti pairings on a metric space and prove a number of approximation results. Finally, in Section \ref{sec:leastgradient}, we provide a characterisation of solutions to the least gradient problem in metric measure spaces, using the fact that existence of the first-order differential structure enables us to use the classical methods of duality theory.

\section{Preliminaries}\label{sec:preliminaries}

\subsection{Sobolev and BV functions in metric measure spaces}
	
Let $(\mathbb{X}, d, \nu)$ be a metric measure space. For any $p \in [1,\infty)$, in the literature there are several possible definitions of Sobolev spaces on $\mathbb{X}$, most prominently via $p$-upper gradients, $p$-relaxed slopes, and via test plans. On complete and separable metric spaces equipped with a doubling measure  (or even under slightly weaker assumptions), all these definitions agree (see \cite{AGS1,DiMarinoTh}); since in this paper we will work under these assumptions, we will choose the most suitable definition for our purposes: the Newtonian spaces. We follow the presentation in \cite{BB}.

\begin{definition}{\rm We say that a measure $\nu$ on a metric space $\mathbb{X}$ is {\it doubling}, if there exists a constant $C_d \geq 1$ such that following condition holds:
\begin{equation}
0 < \nu(B(x,2r)) \leq C_d \, \nu(B(x, r)) < \infty
\end{equation}
for all $x \in \mathbb{X}$ and $r > 0$. The constant $C_d$ is called the doubling constant of $\mathbb{X}$. }
\end{definition}

\begin{definition}{\rm  We say that $\mathbb{X}$ supports a {\it weak $(1,p)$-Poincar\'{e} inequality} if there exist constants $C_P > 0$ and $\lambda \geq 1$ such that for all balls $B \subset \mathbb{X}$, all measurable functions $f$ on $\mathbb{X}$ and all upper gradients $g$ of $f$,
$$\dashint_B \vert f - f_B \vert d \nu \leq C_P r \left( \dashint_{\lambda B} g^p d\nu \right)^{\frac1p},  $$
where $r$ is the radius of $B$ and
$$f_B:= \dashint_B f d\nu := \frac{1}{\nu(B)} \int_B f d\nu.$$

 As a consequence of H\"older's inequality, whenever $\mathbb{X}$ supports a  weak $(1,p)$-Poincar\'{e} inequality, then it supports a weak $(1,q)$-Poincar\'{e} inequality for every $q \geq p$.
}
\end{definition}

We say that a Borel function $g$ is an {\it upper gradient} of a Borel function $u: \mathbb{X} \rightarrow \R$ if for all curves $\gamma: [0,l_\gamma] \rightarrow \mathbb{X}$ we have
$$ \left\vert u(\gamma(l_\gamma)) - u(\gamma(0)) \right\vert \leq \int_\gamma g \,  ds := \int_0^{l_\gamma} g(\gamma(t))  \, \vert \dot{\gamma}(t) \vert \, dt \,  ds,$$
where $$\vert \dot{\gamma}(t) \vert:= \lim_{\tau \to 0} \frac{\gamma(t + \tau) - \gamma(t)}{\tau}$$
is the {\it metric speed} of $\gamma$.

If this inequality holds for $p$-almost every curve, i.e. the $p$-modulus (see for instance \cite[Definition 1.33]{BB}) of the family of all curves for which it fails equals zero, then we say that $g$ is a {\it $p$-weak upper gradient} of $u$.

The Sobolev-Dirichlet class $D^{1,p}(\mathbb{X})$ consists of all Borel functions $u: \mathbb{X} \rightarrow \R$ for which there exists  an upper gradient (equivalently: a $p$-weak upper gradient) which lies in $L^p(\mathbb{X},\nu)$. The Sobolev space $W^{1,p}(\mathbb{X}, d, \nu)$ is defined as
$$W^{1,p}(\mathbb{X}, d, \nu):= D^{1,p}(\mathbb{X}) \cap L^p(\mathbb{X}, \nu).$$
In the literature, this space is sometimes called the Newton-Sobolev space (or Newtonian space) and is denoted $N^{1,p}(\mathbb{X})$. The space $W^{1,p}(\mathbb{X},d,\nu)$ is endowed with the norm
\begin{equation*}
\| u \|_{W^{1,p}(\mathbb{X},d,\nu)} = \bigg( \int_{\mathbb{X}} |u|^p \, d\nu + \inf_g \int_{\mathbb{X}} g^p \, d\nu \bigg)^{1/p},
\end{equation*}
where the infimum is taken over all upper gradients of $u$. Equivalently, we may take the minimum over the set of all $p$-weak upper gradients, see \cite[Lemma 1.46]{BB}. Under the assumptions that $\nu$ is doubling and a weak $(1,p)$-Poincar\'e inequality is satisfied, Lipschitz functions are dense in $W^{1,p}(\mathbb{X},d,\nu)$ (see \cite[Theorem 5.1]{BB}). Let us also stress that the same definition may be applied to open subsets $\Omega \subset \mathbb{X}$.

For every $u \in W^{1,p}(\mathbb{X},d,\nu)$ (even $u \in D^{1,p}(\mathbb{X})$), there exists a minimal $p$-weak upper gradient $|Du| \in L^p(\mathbb{X},\nu)$, i.e. we have $|Du| \leq g$ $\nu$-a.e. for all $p$-weak upper gradients $g \in L^p(\mathbb{X},\nu)$ (see \cite[Theorem 2.5]{BB}). It is unique up to a set of measure zero. In particular, we may simply plug in $|Du|$ in the infimum in the definition of the norm in $W^{1,p}(\mathbb{X},d,\nu)$. Moreover, in \cite{AGS1} (see also \cite{DiMarinoTh}) it was proved that on complete and separable metric spaces equipped with  a nonnegative Borel measure finite on bounded sets not only the various definitions of Sobolev spaces are equivalent, but also that various definitions of $|Du|$ are equivalent, including the Cheeger gradient or the minimal $p$-relaxed slope of $u$.

Recall that for a function $u : \mathbb{X} \rightarrow \R$, its {\it slope} (also called local Lipschitz constant) is defined by
$$\vert \nabla u \vert(x) := \limsup_{y \to x} \frac{\vert u(y) - u(x)\vert}{d(x,y)},$$
with the convention that  $\vert \nabla u \vert(x) = 0$ if $x$ is an isolated point.
	
\begin{remark}\label{curvature} {\rm  For locally Lipschitz functions, it is clear that $\vert Du \vert \leq \vert \nabla u \vert$. In general the equality is not true, but there are two important cases in which we have $\vert Du \vert = \vert \nabla u \vert$ $\nu$-a.e. These are:

(1) When $(\mathbb{X},d, \nu)$ is a metric measure spaces with Riemannian Ricci curvature bounded from below (see \cite	{AGS2});

(2) When $\nu$ is  doubling and $(\mathbb{X},d, \nu)$ supports a weak $(1,p)$-Poincar\'e inequality for some $p > 1$ (see \cite{Chegeer}); in this paper, we will work under these assumptions. $\blacksquare$
			
}
\end{remark}

As in the case of Sobolev functions, in the literature there are several different ways to characterise the total variation in metric measure spaces. However, on complete and separable metric spaces equipped with a doubling measure (or even under a bit weaker assumptions), these notions turn out to be equivalent, see \cite{ADiM} and \cite{DiMarinoTh}. In this paper, we will employ the definition of total variation introduced by Miranda in \cite{Miranda1}.  For $u \in L^1(\mathbb{X},\nu)$, we define the total variation of $u$ on an open set $\Omega \subset \mathbb{X}$ by the formula
\begin{equation}\label{dfn:totalvariationonmetricspaces}
\vert D u \vert_{\nu}(\Omega):= \inf \left\{ \liminf_{n \to \infty} \int_\Omega g_{u_n} \, d\nu \ : \ u_n \in Lip_{loc}( \Omega), \ u_n \to u \ \hbox{in} \ L^1(\Omega, \nu) \right\},
\end{equation}
 where $g_{u_n}$ is a $1$-weak upper gradient of $u$ (we may take $g_{u_n} = |\nabla u_n|$, see \cite{ADiM}). Under the assumptions that $\nu$ is doubling on $\Omega$ and $(\Omega,d,\nu)$ satisfies a weak $(1,1)$-Poincar\'e inequality, since by \cite[Theorem 5.1]{BB} Lipschitz functions are dense in $W^{1,1}(\Omega,d,\nu)$, in the definition above we may require that $u_n$ are Lipschitz functions instead of locally Lipschitz functions. Moreover, the total variation $|Du|_\nu(\mathbb{X})$ defined by formula \eqref{dfn:totalvariationonmetricspaces} is lower semicontinuous with respect to convergence in $L^1(\mathbb{X},\nu)$.

The space of functions of bounded variation $BV(\mathbb{X},d,\nu)$ consists of all functions $u \in L^1(\mathbb{X},\nu)$ such that $|Du|_\nu(\mathbb{X}) < \infty$. It is a Banach space with respect to the norm
$$\Vert u \Vert_{BV(\mathbb{X},d,\nu)}:= \Vert u \Vert_{L^1(\mathbb{X},\nu)} + \vert D u \vert_{\nu}(\mathbb{X}).$$
Convergence in norm is often too much to ask when we deal with $BV$ functions, therefore we will employ the notion of strict convergence. We say that a sequence $\{ u_i \} \subset BV(\mathbb{X},d,\nu)$ {\it strictly converges} to $u \in BV(\mathbb{X},d,\nu)$, if $u_i \to  u$ in $L^1(\mathbb{X}, \nu)$ and $| Du_i |_\nu(\mathbb{X}) \to | Du |_\nu(\mathbb{X})$.  In the course of the paper, we will also employ these definitions for subsets of $\mathbb{X}$ (typically an open set $\Omega$ or its closure), which can be viewed as metric measure spaces equipped with a restriction of the metric $d$ and the measure $\nu$.

In the setting of metric measure spaces, there are a few possible ways to define the boundary measure of an open set. The most standard one is via the definition of $BV$ spaces. A set $E \subset \mathbb{X}$ is said to be of finite perimeter if $\1_E \in BV(\mathbb{X},d, \nu)$, and its perimeter is defined as
$${\rm Per}_{\nu}(E):= \vert D \1_E \vert_{\nu}(\mathbb{X}).$$
If $U \subset \mathbb{X}$ is an open set, we define the perimeter of $E$ in $U$ as
$${\rm Per}_{\nu}(E, U):= \vert D \1_E \vert_{\nu}(U).$$

An alternate definition was given in \cite{KLLS}. Its virtue is that the structure of $
\mathbb{X}$ outside $\overline{\Omega}$ does not enter the definition, because we only allow the approximating sequence of Lipschitz functions to be nonzero in $\Omega$.

\begin{definition}\label{dfn:innerperimeter} {\rm Given open sets $\Omega, U$ in $\mathbb{X}$, we define  the {\it inner perimeter} of $\Omega$ in $U$ as
$$P_+(\Omega,U):= \inf \left\{ \liminf_{n \to \infty} \int_U g_{\psi_n} d\nu \right\},$$
where each $g_{\psi_n}$ is the minimal $1$-weak upper gradient of $\psi_n$, and where the infimum is taken over all sequences $(\psi_n) \subset {\rm Lip}_{loc}(U)$ such that $\psi_n \to \1_\Omega$ in $L^1(U, \nu)$ and $\psi_n = 0$ in $U \setminus \Omega$ for each $n \in \N$.

Furthermore, for any $A \subset \mathbb{X}$ we let
$$P_+(\Omega,A):= \inf \left\{ P_+(\Omega,U) \ : \ U \ \hbox{open}, \, A \subset U \right\}.$$
If $P_+(\Omega,\mathbb{X}) < \infty$, then $P_+(\Omega,\cdot)$ is a Radon measure on $\mathbb{X}$, which we call the {\it inner perimeter measure} of $\Omega$.  Under an additional assumption on $\Omega$ called the exterior measure density condition, the perimeter $Per_\nu$ and the inner perimeter $P_+$ are equivalent, see \cite[Theorem 6.9]{KLLS}; later in the paper, we will prove that for regular domains (see Definition \ref{def:regulardomain}) these notions coincide.}
\end{definition}

The third common way to define the boundary measure in metric measure spaces in the {\it codimension one Hausdorff measure}. Given a set $A \subset \mathbb{X}$, it is defined as
\begin{equation}\label{eq:hausdorffcodimone}
\mathcal{H}(A):= \lim_{R \to 0} \inf \left\{ \sum_{i=1}^\infty \frac{\nu(B(x_i,r_i))}{r_i} \ : \ A \subset \bigcup_{i=1}^\infty B(x_i,r_i), \ 0 < r_i \leq R \right\}.
\end{equation}
It is known from \cite[Theorem 5.3]{Ambrosio} that if $E \subset \mathbb{X}$ is of finite perimeter, then for any Borel set $A \subset \mathbb{X}$,
\begin{equation}\label{Ammb}
\frac{1}{C} \mathcal{H} (A \cap \partial_* E) \leq {\rm Per}_{\nu}(E, A) \leq C \mathcal{H} (A \cap \partial_* E),
\end{equation}
where $\partial_* E$ is the measure theoretical boundary of $E$, that is, the collection of all points $x \in \mathbb{X}$ for which simultaneously
$$\limsup_{r \to 0^+} \frac{\nu(B(x,r) \cap E)}{\nu(B(x,r))} >0, \quad  \limsup_{r \to 0^+} \frac{\nu(B(x,r) \setminus E)}{\nu(B(x,r))} >0.$$
In particular, if $\partial_*\Omega = \partial\Omega$, the spaces $L^p(\partial\Omega,\mathcal{H})$, $L^p(\partial\Omega, |D\1_{\Omega}|)$ coincide as sets for every $p \in [1,\infty]$, and are equipped with equivalent norms. We will write explicitly which norm we use every time where it is not clear from the context.

Definition of boundary values of BV functions in a metric measure space is a more delicate issue. We will restrict our attention to open sets and adopt the following definition, used for instance in \cite{LS} and \cite{KLLS}:
	
\begin{definition}\label{dfn:trace}
{\rm Let $\Omega \subset \mathbb{X}$ be an open set and let $u$ be a $\nu$-measurable function on $\Omega$. A number $T_\Omega u(x)$ is a {\it trace} of $u$ at $x \in \partial \Omega$ if
$$\lim_{r \to 0^+} \dashint_{\Omega \cap B(x,r)} \vert u - T_\Omega u(x) \vert \, d \nu = 0.$$
We say that $u$ has a trace in $\partial \Omega$ if $ T_\Omega u(x)$ exists for $\mathcal{H}$-almost every $x \in \partial \Omega$.}
\end{definition}

Well-posedness of the trace and identifying the trace space of $W^{1,1}(\Omega,d,\nu)$ or $BV(\Omega,d,\nu)$ in the setting of metric measure spaces is not immediate and requires additional structural assumptions on $\Omega$. We summarise the results known in the literature in the following Theorem, which is a combination of \cite[Theorem 1.2]{MSS} and \cite[Theorem 5.5]{LS}.

\begin{theorem}\label{thm:traces}
Suppose that $\nu$ is doubling and $\mathbb{X}$ supports a weak $(1,1)$-Poincar\'e inequality. Let $\Omega$ be an open bounded set which supports a weak $(1,1)$-Poincar\'e inequality. Assume that $\Omega$ additionally satisfies the {\it measure density condition}, i.e.
there is a constant $C > 0$ such that
\begin{equation}\label{MDC1}
\nu(B(x,r) \cap \Omega) \geq C \nu(B(x,r))
\end{equation}
for $\mathcal{H}$-a.e. $x \in \partial \Omega$ and every $r \in (0, {\rm diam}(\Omega))$. Moreover, assume that {\it $\partial\Omega$ is Ahlfors codimension 1 regular}, i.e. there is a constant $C > 0$ such that
\begin{equation}\label{eq:codimensiononeregularity}
C^{-1} \frac{\nu(B(x,r))}{r} \leq \mathcal{H}(B(x,r) \cap \partial\Omega) \leq C \frac{\nu(B(x,r))}{r}
\end{equation}
for all $x \in \partial\Omega$ and every $r \in (0,\mbox{diam}(\Omega))$. \\
\\
Under these assumptions, Definition \ref{dfn:trace} defines an operator $T_\Omega: BV(\Omega,d,\nu) \twoheadrightarrow L^1(\partial\Omega,\mathcal{H})$. Moreover, the operator $T_\Omega$ is linear, bounded and surjective.
\end{theorem}

Furthermore, under the same assumptions there is a (nonlinear) bounded extension operator $\mbox{Ext}: L^1(\partial\Omega,\mathcal{H}) \rightarrow BV(\Omega,d,\nu)$ such that $T_\Omega \circ \mbox{Ext}$ is the identity operator on $L^1(\partial\Omega,\mathcal{H})$. As the discussion in \cite{MSS} shows, any of these conditions cannot be dropped if we want the trace operator to be surjective; however, in order for it to be a linear and bounded operator with values in $L^1(\partial\Omega,\mathcal{H})$ we may weaken the assumptions a little bit and only assume that the upper bound in \eqref{eq:codimensiononeregularity} holds.

The measure density condition \eqref{MDC1} has an important consequence. Namely, let $f = \1_{\mathbb{X} \backslash \Omega} \in L^\infty(\mathbb{X},\nu)$. By the Lebesgue differentiation theorem, $\nu$-a.e. point in $\mathbb{X} \backslash\Omega$ has density one. But this is impossible for $\mathcal{H}$-a.e. $x \in \Omega$; moreover, by definition of $\mathcal{H}$, the set of measure zero with respect to $\mathcal{H}$ is also a set of measure zero with respect to $\nu$. Hence, $\nu(\partial\Omega) = 0$. Moreover, when $\mathbb{X}$ is complete, the assumption that $\partial\Omega$ is Ahlfors codimension 1 regular also has an important consequence: $\partial\Omega$ is a closed bounded subset, so it is compact (see \cite[Proposition 3.1]{BB}). Then, we cover it with a finite number of balls and use the estimate from above in \eqref{eq:codimensiononeregularity} to conclude that $\mathcal{H}(\partial\Omega) < \infty$.

In the course of the paper, we will sometimes consider BV or Sobolev functions defined on the closure of $\Omega$, which can be viewed as a metric measure space in its own right. This is done for two principal reasons. One is that equivalence results between various definitions of Sobolev spaces require the metric space to be complete. The second one, perhaps more important, is that the construction of the first order differential structure introduced by Gigli (see \cite{Gig}) used extensively in the second part of the paper requires the metric space to be complete. The Remark below states that under the assumptions of Theorem \ref{thm:traces} these concerns are largely irrelevant, since we may identify the Sobolev and BV spaces on $\Omega$ and $\overline{\Omega}$.  In fact, throughout most of the paper we will work under the assumptions of Theorem \ref{thm:traces}.

\begin{remark}\label{rem:extensiontoclosure}{\rm
Suppose that $(\mathbb{X},d)$ is complete and separable and $\Omega \subset \mathbb{X}$ is an open bounded set. Then, $(\overline{\Omega},d)$ is also a complete and separable metric space. Assume that $\nu(\partial\Omega) = 0$. Then, if $\nu$ is doubling on $\Omega$ and $(\Omega,d,\nu)$ satisfies a weak $(1,p)$-Poincar\'e inequality, then $\nu$ is doubling on $\overline{\Omega}$ and $(\overline{\Omega},d,\nu)$ satisfies a weak $(1,p)$-Poincar\'e inequality, see \cite[Proposition 7.1]{AS}. Under these assumptions, the Sobolev and $BV$ spaces defined on $\Omega$ and $\overline{\Omega}$ coincide:

(1) For every function in $W^{1,p}(\overline{\Omega},d,\nu)$, its restriction to $\Omega$ lies in $W^{1,p}(\Omega,d,\nu)$. On the other hand, every $u \in W^{1,p}(\Omega,d,\nu)$ has an extension $\overline{u} \in W^{1,p}(\overline{\Omega},d,\nu)$ such that
\begin{equation*}
\| u \|_{W^{1,p}(\Omega,d,\nu)} = \| \overline{u} \|_{W^{1,p}(\overline{\Omega},d,\nu)},
\end{equation*}
see \cite[Lemma 8.2.3]{HKST}. Since by \cite[Theorem 5.1]{BB} Lipschitz functions are dense in both $W^{1,p}(\Omega,d,\nu)$ and $W^{1,p}(\overline{\Omega},d,\nu)$, we may conclude that $W^{1,p}(\Omega,d,\nu) = W^{1,p}(\overline{\Omega},d,\nu)$, see \cite[Proposition 7.1]{AS}.

(2) For every function in $BV(\overline{\Omega},d,\nu)$, its restriction to $\Omega$ lies in $BV(\Omega,d,\nu)$. On the other hand, every $u \in BV(\Omega,d,\nu)$ has an extension $\overline{u} \in BV(\overline{\Omega},d,\nu)$ such that
\begin{equation}\label{eq:bvextensiontoclosure}
\| u \|_{BV(\Omega,d,\nu)} = \| \overline{u} \|_{BV(\overline{\Omega},d,\nu)},
\end{equation}
see \cite[Proposition 3.3]{LS}.

In fact, the spaces $BV(\Omega,d,\nu)$ and $BV(\overline{\Omega},d,\nu)$ are isometric. To see this, notice that the map $\overline{u} \mapsto \overline{u}|_{\Omega}$ from $BV(\overline{\Omega},d,\nu)$ to $BV(\Omega,d,\nu)$ is linear and by equation \eqref{eq:bvextensiontoclosure} it is surjective. It remains to be shown that the norm is preserved; then, this map is in fact bijective
and has an inverse which also is a linear isometry. Clearly $\| \overline{u} \|_{L^1(\overline{\Omega},\nu)} = \| \overline{u}|_\Omega \|_{L^1(\Omega,\nu)}$; we will now see that the total variations coincide. Let $u_n \in \mbox{Lip}(\overline{\Omega})$ be a sequence of Lipschitz functions converging strictly to $\overline{u}$ (we may assume it to be Lipschitz by \cite[Theorem 5.1]{BB}, see the comment under the definition of the total variation). Then
\begin{equation*}
|D\overline{u}|_\nu(\overline{\Omega}) = \lim_{n \rightarrow \infty} \int_{\overline{\Omega}} g_{u_n} \, d\nu = \lim_{n \rightarrow \infty} \int_{\Omega} g_{u_n} \, d\nu \geq |D(\overline{u}|_\Omega)|_\nu(\Omega),
\end{equation*}
since the sequence $(u_n|_\Omega)$ converges in $L^1(\Omega,\nu)$ to $\overline{u}|_\Omega$. Here, $|D\overline{u}|_\nu(\overline{\Omega})$ is understood as the total variation on the whole measure space $(\overline{\Omega},d,\nu)$ and not as the total variation on a subset of $\mathbb{X}$.

On the other hand, let $v_n \in \mbox{Lip}(\Omega)$ be a sequence of Lipschitz functions converging strictly to $\overline{u}|_\Omega$. Denote by $\overline{v_n} \in \mbox{Lip}(\overline{\Omega})$  the Lipschitz extension of $v_n$ to $\overline{\Omega}$ (given by the McShane construction). Then, $\overline{v_n}$ converges in $L^1(\overline{\Omega},\nu)$ to $\overline{u}$, so
\begin{equation*}
|D(\overline{u}|_\Omega)|_\nu(\Omega) = \lim_{n \rightarrow \infty} \int_{\Omega} g_{v_n} \, d\nu = \lim_{n \rightarrow \infty} \int_{\overline{\Omega}} g_{\overline{v_n}} \, d\nu \geq |D\overline{u}|_\nu(\overline{\Omega}),
\end{equation*}
where in the second equality we used \cite[Lemma 6.3.8]{HKST} to conclude that we may require $g_{v_n} = g_{\overline{v_n}}$ on $\Omega$. Hence, the spaces $BV(\Omega,d,\nu)$ and $BV(\overline{\Omega},d,\nu)$ are isometric. $\blacksquare$}
\end{remark}

In the statement of Remark \ref{rem:extensiontoclosure}, instead of considering $\Omega \subset \mathbb{X}$ with $\nu(\partial\Omega) = 0$ and its closure $\overline{\Omega}$ we may consider a noncomplete metric space $Z$ and its metric completion $\hat{Z}$. Then, the Remark is valid if we equip $\hat{Z}$ with the null-extension of $\nu$; in fact, this is the language in which \cite[Proposition 7.1]{AS} and \cite[Lemma 3.3]{LS} were originally formulated.

\subsection{The differential structure}\label{subsec:diffstructure}

In the next Sections, we will first prove a version of the Gauss-Green formula on regular domains in metric measure spaces which is valid for $BV$ functions instead of Lipschitz functions, and use it to provide a characterisation of least gradient functions (or, in other words, the subdifferential of the $1$-Laplacian). Our main tool will be the linear differential structure on a metric measure space $(\mathbb{X},d,\nu)$ introduced by Gigli. We follow Gigli (see \cite{Gig}) and Buffa-Comi-Miranda (see \cite{BCM}) in the introduction of this first-order differential structure.

From now on, we assume that $\mathbb{X}$ is a complete and separable metric space and $\nu$ is a nonnegative Radon measure. We will introduce additional assumptions on $(\mathbb{X},d,\nu)$ in due course. In particular,  in order to
introduce the Anzellotti pairings on $(\mathbb{X},d,\nu)$ and use them to prove a Gauss-Green formula valid for $BV$ functions, we will require the doubling and Poincar\'e assumptions. Furthermore, to characterise the subdifferential of the $1$-Laplacian we will require that the assumptions of Theorem \ref{thm:traces} are satisfied, because we need to work with approximations which preserve traces.

\begin{definition}{\rm
We define the cotangent module to $\mathbb{X}$ as
$$ \mbox{PCM}_p = \left\{ \{(f_i, A_i)\}_{i \in \N} \ : \ (A_i)_{i \in\N} \subset \mathcal{B}(\mathbb{X}), \  f_i \in D^{1,p}(A_i), \ \ \sum_{i \in \N} \int_{A_i} |Df_i|^p \, d\nu < \infty  \right\},$$
where $A_i$ is a partition of $\mathbb{X}$. We define the equivalence relation $\sim$ as
$$\{(A_i, f_i)\}_{i \in \N} \sim \{(B_j,g_j)\}_{j \in \N} \quad \mbox{if} \quad |D(f_i - g_j)| = 0 \ \ \nu-\hbox{a.e. on} \ A_i \cap B_j.$$
Consider the map $\vert \cdot \vert_* : \mbox{PCM}_p/\sim \, \rightarrow L^p(\mathbb{X}, \nu)$ given by
$$\vert \{(f_i, A_i)\}_{i \in \N} \vert_* := \vert D f_i \vert$$
$\nu$-everywhere on $A_i$ for all $i \in \N$, namely the {\it pointwise norm} on $ \mbox{PCM}_p/\sim$.

In $ \mbox{PCM}_p/\sim$ we define the norm $\| \cdot \|$ as
$$ \|  \{(f_i, A_i)\}_{i \in \N} \|^p = \sum_{i\in \N} \int_{A_i}|Df_i|^p $$
and set $L^p(T^* \mathbb{X})$ to be the closure of  $\mbox{PCM}_p / \sim$ with respect to this norm, i.e. we identify functions which differ by a constant and we identify possible rearranging of the sets $A_i$. $L^p(T^* \mathbb{X})$ is called the {\it cotangent module} and its elements will be called {\it $p$-cotangent vector field}.

$L^p(T^* \mathbb{X})$ is a $L^p(\nu)$-normed module (see \cite{Gig} for the theory of $L^p(\nu)$-normed modules). We denote by $\vert \cdot \vert_*$  the pointwise norm on $L^p(T^* \mathbb{X})$, i.e., $\vert \cdot \vert_* : L^p(T^* \mathbb{X}) \rightarrow L^p(\mathbb{X}, \nu)$ such that
$$\Vert \vert v \vert_* \Vert_{L^p(\mathbb{X}, \nu)} = \Vert v \Vert_{L^p(T^* \mathbb{X})}, \quad \vert f v \vert_* = \vert f \vert \vert v \vert_*, \ \ \nu-a.e.,$$
for every $v \in L^p(T^* \mathbb{X})$ and $f \in L^\infty (\mathbb{X},\nu)$.  This second property is called the $L^\infty$-linearity of the cotangent module, see \cite{Gig} or \cite[Appendix A]{BCM}.

We denote by  $L^q(T\mathbb{X})$ the dual module of $L^p(T^* \mathbb{X})$, namely $$L^q(T\mathbb{X}):= \hbox{HOM}(L^p(T^* \mathbb{X}), L^1(\mathbb{X}, \nu)),$$
i.e., is a bounded linear operator from $L^p(T^* \mathbb{X})$ to $L^1(\mathbb{X}$ viewed as Banach spaces and further satisfies the locality condition
$$ X(f \omega) = f \cdot X(\omega), \quad \forall \omega \in L^p(T^* \mathbb{X}), \ \ f \in L^\infty(\mathbb{X}, \nu).$$
We have that $L^q(T\mathbb{X})$ is a $L^q(\nu)$-normed module (see \cite{Gig}),  it is $L^\infty$-linear and we will denote by $\vert \cdot \vert$ its pointwise norm. The elements of $L^q(T\mathbb{X})$ will be called  {\it $q$-vector fields} on $\mathbb{X}$. The duality between $\omega \in L^p(T^* \mathbb{X})$ and $X \in  L^q(T\mathbb{X})$ will be denoted by $\omega(X) \in L^1(\mathbb{X}, \nu)$.  Since $L^p(T^* \mathbb{X})$ is reflexive as a module (for all $p \in [1,\infty]$) we can identify
$$ L^q(T\mathbb{X})^* = L^p(T^*\mathbb{X}),$$
where $\frac{1}{p} + \frac{1}{q} = 1$.}
\end{definition}
	
\begin{definition}\label{dfn:differential}
{\rm
Given $f \in D^{1,p}(\mathbb{X})$ we can define its {\it differential} $df$ as an element of $L^p(T^* \mathbb{X})$ given by the formula $df = (f, \mathbb{X})$. Moreover,
$$\vert df \vert_* = \vert D f \vert, \quad \nu-a.e. \quad \forall f \in D^{1,p}(\mathbb{X}).$$}

\end{definition}

Clearly, the operation of taking the differential is linear as an operator from $D^{1,p}(\mathbb{X})$ to $L^p(T^{*} \mathbb{X})$; moreover, from the definition of the norm in $L^p(T^{*} \mathbb{X})$ it is clear that this operator is bounded with norm equal to one. Furthermore, for all $X \in L^q(T\mathbb{X})$ and $f \in D^{1,p}(\mathbb{X})$ we have
\begin{equation*}
df(X) \leq |df|_* |X| \leq \frac1p |df|_*^p + \frac1q |X|^q \qquad \nu-\mbox{a.e.}
\end{equation*}
Let us note that when the metric measure space is $(\mathbb{R}^N,d_{Eucl}, \mathcal{L}^N)$, the vector fields and differentials arising from this construction coincide with their standard counterparts defined in coordinates, see \cite[Remark 2.2.4]{Gig}.

Now, we define the divergence of a vector field, in the case when it can be represented by an $L^p$ function or a Radon measure. Following \cite{BCM,DiMarinoTh}, we set
$$ \mathcal{D}^q(\mathbb{X}) = \left\{ X \in L^q(T\mathbb{X}): \, \exists f \in L^q(\mathbb{X},\nu) \,\, \int_\mathbb{X} fg d\nu = - \int_{\mathbb{X}} dg(X) d\nu \ \ \forall g \in  W^{1,p}(\mathbb{X},d,\nu) \right\}. $$
Here, the right hand side makes sense as an action of an element of $L^p(T^* \mathbb{X})$ on an element of $L^{q}(T\mathbb{X})$; the resulting function is an element of $L^1(\mathbb{X},\nu)$. The function $f$, which is unique by the density of  $W^{1,p}(\mathbb{X},d,\nu)$ in $ L^{p}(\mathbb{X},\nu)$, will be called the {\it $q$-divergence} of the vector field $X$, and we shall write $\mbox{div}(X)= f$.  The dependence of the divergence on $q$ is discussed at length in \cite{BCM}.

 In the course of the paper, in order to prove the generalised Gauss-Green formula, we also need to consider the case when $X \in L^\infty(T\mathbb{X})$, but its divergence lies in $L^r(\mathbb{X},\nu)$. To this end, let us recall the following space introduced in \cite{GM2021}. For $\frac1r + \frac1s = 1$, we set
$$ \mathcal{D}^{q,r}(\mathbb{X}) = \bigg\{ X \in L^q(T\mathbb{X}): \,\, \exists f \in L^r(\mathbb{X},\nu) \quad \forall g \in  W^{1,p}(\mathbb{X},d,\nu) \cap L^s(\mathbb{X},\nu) \qquad $$
$$ \qquad\qquad\qquad\qquad\qquad\qquad\qquad\qquad\qquad\qquad\qquad \int_\mathbb{X} fg \, d\nu = - \int_{\mathbb{X}} dg(X) \, d\nu \bigg\}. $$
This uniquely defined function $f$ will be called the $(q,r)$-divergence of $X$. We will still write $\mbox{div}(X) = f$ when it is clear from the context. Whenever Lipschitz functions are dense in $W^{1,p}(\mathbb{X},d,\nu)$ (as in \cite[Theorem 5.1]{BB}), then the divergence does not depend on $r$ in the following sense: if $f$ is the $(q,r)$-divergence of $X$ and $f \in L^{r'}(\mathbb{X},\nu)$, then it is also the $(q,r')$-divergence of $X$. Also, note that $\mathcal{D}^{q,q}(\mathbb{X}) = \mathcal{D}^{q}(\mathbb{X})$. In this paper, we will be solely interested in the case $q = \infty$; for the generalised Gauss-Green formula, we will consider arbitrary $r$, but actually for the study of least gradient functions it is sufficient to consider the case $q = r = \infty$.

Finally, let us note that it is also possible to define the divergence in a similar fashion also in the case when it is represented by a Radon measure; following \cite{BCM}, we set
$$ \mathcal{DM}^q(\mathbb{X}) = \left\{ X \in L^q(T\mathbb{X}): \, \exists \, \mu \in \mathcal{M}(\mathbb{X}) \,\, \int_\mathbb{X} g \, d\mu = - \int_{\mathbb{X}} dg(X) \, d\nu \ \ \forall g \in \mbox{Lip}_{bs}(\mathbb{X}) \right\}. $$
We still write $\mbox{div}(X) = \mu$.

Using the Leibniz rule for the differential, see \cite[Corollary 2.2.8]{Gig}, it is easy to see that whenever $X \in \mathcal{D}^q(\mathbb{X})$ and $f \in L^\infty(\mathbb{X},\nu) \cap D^{1,p}(\mathbb{X})$ with $\vert D f \vert \in L^\infty(\mathbb{X},\nu)$, we have
\begin{equation*}
fX \in \mathcal{D}^q(\mathbb{X}) \quad \hbox{and} \quad \mbox{div}(fX) = df(X) + f \mbox{div}(X).
\end{equation*}
Furthermore, whenever $X \in L^\infty(T\mathbb{X})$ with $\mbox{div}(X) \in L^q(\mathbb{X},\nu)$, we have
\begin{equation}\label{eq:leibnizformula}
fX \in L^\infty(T\mathbb{X}), \quad \mbox{div}(fX) \in L^q(\mathbb{X},\nu) \quad \hbox{and} \quad \mbox{div}(fX) = df(X) + f \mbox{div}(X).
\end{equation}

In the course of the paper, we will extensively rely on the first order differential structure presented above. It is well-defined on metric spaces which are complete and separable. A priori, the structure is not defined locally - the objects $T^* \mathbb{X}$ and $T \mathbb{X}$ are not necessarily well-defined (the notation $L^p(T^* \mathbb{X})$ and $L^q(T\mathbb{X})$ is purely formal) and it is not immediate how to localise it to an open set $\Omega \subset \mathbb{X}$ (nonetheless, some positive results on pointwise identification of the tangent and cotangent modules can be found in \cite{LP} and \cite{LPR}). However, whenever $\Omega \subset \mathbb{X}$ is an open bounded set, then $\overline{\Omega}$ is also a complete and separable metric space; hence, the whole first-order differential structure described above may be defined on $\overline{\Omega}$ as well. Under the assumptions of Remark \ref{rem:extensiontoclosure}, we may identify Newton-Sobolev functions on $\Omega$ and $\overline{\Omega}$. Then, on $\overline{\Omega}$ the Newton-Sobolev space is equivalent to the Sobolev space defined by test-plans as in \cite{Gig} and \cite{BCM}, see \cite[Theorem B.4]{Gig2}. Using this identification, we may also define the differential structure on $\Omega$ if it is sufficiently regular; with a slight abuse of notation, we write $L^p(T^* \Omega)$, $L^q(T\Omega)$, $\mathcal{D}^q(\Omega)$  and $\mathcal{D}^{q,r}(\Omega)$, even though technically these objects are defined via an isometric extension to $\overline{\Omega}$.  However, note that $\mathcal{DM}^q(\Omega)$ cannot be defined in this way; for a vector field in $\mathcal{DM}^q(\overline{\Omega})$, its divergence may be a measure which gives mass to the boundary $\partial\Omega$. Indeed, this is very often the case, see the comment to Theorem \ref{thm:bufacomimiranda}.

However, with this understood, the definition of the divergence introduced above is not suitable for our purposes. Indeed, on a bounded domain, the definition of the divergence as above takes into account the boundary effects. In order to prove a relaxed version of the Gauss-Green formula using Anzellotti pairings in Section \ref{sec:Anzellotti}, we need to have a notion of divergence which will only see the structure of $X$ inside the open set $\Omega$ (even though it is not clearly defined locally). The solution we suggest is to test the definition of the divergence using only functions which vanish at the boundary. Given an open bounded set $\Omega \subset \mathbb{X}$ which satisfies the assumptions of Theorem \ref{thm:traces}, we set
$$ \mathcal{D}_0^q(\Omega) = \left\{ X \in  L^q(T\Omega): \, \exists f \in L^q(\Omega,\nu) \,\, \int_\Omega fg d\nu = - \int_{\Omega} dg(X) d\nu \ \ \forall g \in  W^{1,p}_0(\Omega,d,\nu) \right\},$$
where $W^{1,p}_0(\Omega,d,\nu)$ is the space of Sobolev functions in $W^{1,p}(\Omega,d,\nu)$ with zero trace. We again say that the (uniquely defined) function $f$ is the divergence of $X$ (when it is clear from the context) and we write $\mbox{div}_0(X) = f$.

Similarly, for $\frac1r + \frac1s = 1$, we set
$$ \mathcal{D}_0^{q,r}(\Omega) = \bigg\{ X \in L^q(T\Omega): \,\, \exists f \in L^r(\Omega,\nu) \quad \forall g \in  W_0^{1,p}(\Omega,d,\nu) \cap L^s(\Omega,\nu) \qquad $$
$$ \qquad\qquad\qquad\qquad\qquad\qquad\qquad\qquad\qquad\qquad\qquad \int_\Omega fg \, d\nu = - \int_{\Omega} dg(X) \, d\nu \bigg\}. $$
We still write $\mbox{div}_0(X) = f$.  Since under the assumptions of Theorem \ref{thm:traces} Lipschitz functions are dense in $W^{1,p}(\mathbb{X},d,\nu)$, the divergence does not depend on $r$ in the following sense: if $f$ is the $(q,r)$-divergence of $X$ and $f \in L^{r'}(\mathbb{X},\nu)$, then it is also the $(q,r')$-divergence of $X$.

Furthermore, the divergence $\mbox{div}_0$ also has property \eqref{eq:leibnizformula}, i.e. whenever $X \in \mathcal{D}_0^q(\Omega)$ and $f \in L^\infty(\Omega,\nu) \cap D^{1,p}(\Omega)$ with $|Df| \in L^\infty(\Omega,\nu)$, we have
\begin{equation*}
fX \in \mathcal{D}_0^q(\Omega) \quad \hbox{and} \quad \mbox{div}_0(fX) = df(X) + f \mbox{div}_0(X),
\end{equation*}
and whenever $X \in L^\infty(T\Omega)$ with $\mbox{div}_0(X) \in L^q(\Omega,\nu)$, we have
\begin{equation}\label{eq:leibnizformulav2}
fX \in L^\infty(T\Omega), \quad \mbox{div}_0(fX) \in L^q(\Omega,\nu) \quad \hbox{and} \quad \mbox{div}_0(fX) = df(X) + f \mbox{div}_0(X).
\end{equation}

We will comment on the relationship between the two definitions of the divergence in Section \ref{sec:Anzellotti}. In short, Theorem \ref{thm:bufacomimiranda} roughly says that the divergence $\mbox{div}(X)$ is the divergence $\mbox{div}_0(X)$ plus a boundary term which has an interpretation of the normal trace.

\section{A general Gauss-Green formula on metric measure spaces}\label{sec:Anzellotti}

Let us first introduce the assumptions on the metric measure space $\mathbb{X}$ that we will use for the remainder of the paper. Suppose that the metric space $(\mathbb{X},d)$ is complete, separable, equipped with a doubling measure $\nu$, and that the metric measure space $(\mathbb{X},d,\nu)$ supports a weak $(1,1)$-Poincar\'e inequality. In particular, these assumptions imply that $\mathbb{X}$ is locally compact, see \cite[Proposition 3.1]{BB}.

Moreover, a subset $\Omega \subset \mathbb{X}$ is always understood to be open and bounded. Furthermore, we suppose that $\nu(\partial\Omega) = 0$ and that $\Omega$ supports a weak $(1,1)$-Poincar\'e inequality. These are exactly the assumptions of Remark \ref{rem:extensiontoclosure}, which allows us to use the first order differential structure on $\Omega$. We will later introduce additional assumptions when necessary; typically, these will be the assumptions required to obtain existence of traces as in Theorem \ref{thm:traces}.

\subsection{A refined approximation result}

In order to define a generalised version of Anzellotti pairings  on open bounded sets, we will need to approximate a $BV$ function by regular enough functions, in the spirit of \cite[Lemma 5.2]{Anz}. Existence of a sequence of locally Lipschitz functions which approximate the desired function in the strict topology is automatic by virtue of Definition \ref{dfn:totalvariationonmetricspaces}, but  it is not enough for the approximation arguments in the definition of Anzellotti pairings and we will require some additional properties of the approximating sequence. The first result we will need is the following Lemma proved in \cite[Lemma 5.1]{GM2021}. Typically, we will use it for $\mathbb{X} = \Omega$, which can be viewed as a metric measure space in its own right.

\begin{lemma}\label{lem:lipschitzapproximation}
Suppose that $u \in BV(\mathbb{X},d, \nu)$. There exists a sequence of Lipschitz functions $u_n \in \mbox{Lip}(\mathbb{X}) \cap BV(\mathbb{X},d,\nu)$ such that: \\
$(1)$ $u_n \rightarrow u$ strictly in $BV(\mathbb{X},d,\nu)$;
\\
$(2)$ Let $p \in [1,\infty)$. If $u \in L^p(\mathbb{X},\nu)$, then $u_n \in L^p(\mathbb{X},\nu)$ and $u_n \rightarrow u$ in $L^p(\mathbb{X},\nu)$;
\\
$(3)$ If $u \in L^\infty(\mathbb{X},\nu)$, then $u_n \in L^\infty(\mathbb{X},\nu)$ and $u_n \rightharpoonup u$ weakly* in $L^\infty(\mathbb{X},\nu)$.
\end{lemma}

In the course of the paper, we will also need a finer approximation: given an appropriate domain $\Omega \subset \mathbb{X}$, we will require additionally that the trace $T_\Omega$ of the approximating sequence is the same as the trace of the desired function. In this case, we may no longer require the approximating sequence to be Lipschitz. In \cite{LS}, it was proved that given $u \in BV(\Omega,d,\nu)$ we can choose an approximating sequence $u_n \in \mbox{Lip}_{loc}(\Omega)$ which converges strictly to $u$ and such that the trace of $u$ is preserved, i.e. $T_\Omega u_n = T_\Omega u$ $\mathcal{H}$-a.e.  The Lemma below is an upgraded version of \cite[Corollary 6.7]{LS}; point $(1)$ is precisely the content of \cite[Corollary 6.7]{LS}, while points $(2)$ and $(3)$ yield additional properties of the approximating sequence. We provide a proof of point $(1)$ for the reader's convenience, since the other points will follow by
making modifications to that proof.

\begin{lemma}\label{lem:goodapproximation}
Under the assumptions of Theorem \ref{thm:traces}, suppose that $u \in BV(\Omega,d, \nu)$. Then, there exist locally Lipschitz functions $u_n \in \mbox{Lip}_{loc}(\Omega) \cap BV(\Omega,d,\nu)$ such that: \\
$(1)$ $u_n \rightarrow u$ strictly in $BV(\Omega,d,\nu)$ and $T_\Omega u_n = T_\Omega u$ $\mathcal{H}$-a.e.;
\\
$(2)$ Let $p \in [1,\infty)$. If $u \in L^p(\Omega,\nu)$, then $u_n \in L^p(\Omega,\nu)$ and $u_n \rightarrow u$ in $L^p(\Omega,\nu)$;
\\
$(3)$ If $u \in L^\infty(\Omega,\nu)$, then $u_n \in L^\infty(\Omega,\nu)$ and $u_n \rightharpoonup u$ weakly* in $L^\infty(\Omega,\nu)$.
\end{lemma}

In the proof, we will rely on the notion of discrete convolutions, which we will shortly describe. There are various definitions in the literature and we will follow the ones described in \cite{HKT,LS}. Fix an open set $\Omega \subset \mathbb{X}$ and a scale $R > 0$. Then, take a {\it Whitney type} covering $\mathcal{B} = \{ B_j = B(x_j, r_j) \}$ of $\Omega$, i.e. a covering which satisfies the following properties (here, $\lambda$ is the constant in the weak $(1,1)$-Poincar\'e inequality):

1. For each $j \in \mathbb{N}$, we have
$$ r_j = \min \{ d(x_j, X \backslash \Omega) / 20\lambda, R \};$$

2. The covering has a bounded overlap property, namely for each $j \in \mathbb{N}$, the ball $5 \lambda B_j$ intersects at most $C_0 = C_0(C_d, \lambda)$ other balls in the covering;

3. If $5 \lambda B_j$ intersects $5 \lambda B_k$, then $r_j \leq 2 r_k$.

Given such a covering, there exists a partition of unity $\phi_j$ subordinate to that covering. Namely, for every $j$ the function $\phi_j$ is $C_0 / r_j$-Lipschitz, $0 \leq \phi_j \leq 1$, $\sum_{j = 1}^{\infty} \phi_j = 1$, and $\mbox{supp}(\phi_j) \subset 2B_j$. Then, for every $u \in L^1(\Omega,\nu)$ we may define its {\it discrete convolution} by the formula
$$ u_{\mathcal{B}} = \sum_{j = 1}^\infty u_{B_j} \phi_j.$$
Then, among other properties, we have that $u_{\mathcal{B}} \in \mbox{Lip}_{loc}(\Omega) \cap L^1(\Omega,\nu)$ and if $u_{\mathcal{B}_n}$ is a sequence of discrete convolutions of $u$ at scales $\frac{1}{n}$, then $u_{\mathcal{B}_n} \rightarrow u$ in $L^1(\Omega,\nu)$. Moreover, $u_{\mathcal{B}}$ admits an upper gradient
\begin{equation}\label{eq:uppergradientfordisconv}
g = C(C_d, C_P, \lambda) \sum_{j=1}^\infty \1_{B_j} \frac{|Du|_\nu(5 \lambda B_j)}{\nu(B_j)}.
\end{equation}

\begin{proof}
(1) For every $\delta > 0$, denote by $\Omega_\delta$ the set
$$ \Omega_\delta = \{ y \in \Omega: d(y, X \backslash \Omega) > \delta \}. $$
Fix $\varepsilon > 0$ and choose $\delta$ small enough so that $|Du|_\nu(\Omega \backslash \Omega_\delta) < \varepsilon$. Then, define
$$ \eta(y) = \max \bigg\{ 0, 1 - \frac{4}{\delta} d(y,\Omega_{\delta/2})  \bigg\}. $$
This function will be used to glue together a strict approximation of $u$ inside $\Omega$ with a discrete convolution near $\partial\Omega$. The function $\eta$ takes values in $[0,1]$ and is Lipschitz with the Lipschitz constant $\frac{4}{\delta}$.

Denote by $v_n \in \mbox{Lip}_{loc}(\Omega)$ the sequence of discrete convolutions of $u$ at scale $\frac{1}{n}$. Denote by $w_n \in \mbox{Lip}_{loc}(\Omega)$ the sequence which converges strictly to $u$ given by \eqref{dfn:totalvariationonmetricspaces}. Then, the function
$$ u_n := \eta w_n + (1 - \eta) v_n $$
is locally Lipschitz in $\Omega$ and it has an $1$-upper gradient (see \cite[Lemma 2.18]{BB})
$$ g_{u_n} := |\nabla \eta| |w_n - v_n| + \eta |\nabla w_n| + (1-\eta) g_{v_n}, $$
where $g_{v_n}$ is the $1$-upper gradient of $v_n$ given by the formula \eqref{eq:uppergradientfordisconv}. Then, we have $u_n \rightarrow u$ in $L^1(\Omega,\nu)$. Moreover, since $0 \leq \eta \leq 1$ and $\eta = 0$ in $\Omega \backslash \Omega_{\delta/2}$, we have
$$ \limsup_{n \rightarrow \infty} \int_\Omega g_{u_n} \, d\nu \leq \limsup_{n \rightarrow \infty} \frac{4}{\delta} \int_\Omega |w_n - v_n| \, d\nu + \limsup_{n \rightarrow \infty} \int_\Omega |\nabla w_n| \, d\nu + \limsup_{n \rightarrow \infty} \int_{\Omega \backslash \Omega_{\delta / 2}} g_{v_n} \, d\nu \leq $$
$$ \leq 0 + |Du|_\nu(\Omega) + C |Du|_\nu(\Omega \backslash \Omega_\delta) \leq |Du|_\nu(\Omega) + C\varepsilon.$$
Moreover,  since assumptions of Theorem \ref{thm:traces} imply that $\mathcal{H}(\partial\Omega) < \infty$, $u_n - u$ has trace zero by virtue of \cite[Proposition 6.5]{LS}. By letting $\varepsilon \rightarrow 0$ and a diagonalisation argument we obtain a subsequence $u_{n_k}$ converging strictly to $u$.

(2) We need to show that the functions $u_n$ constructed above additionally satisfy $u_n \in L^p(\Omega,\nu)$ and $u_n \rightarrow u$ in $L^p(\Omega,\nu)$. Firstly, notice that since $v_n$ are discrete convolutions of $u$, by \cite[Lemma 5.3]{HKT} we have $v_n \in L^p(\Omega,\nu)$ and $v_n \rightarrow u$ in $L^p(\Omega,\nu)$. To end the proof, we need to show that we can modify the strict approximation $w_n$ given by Definition \ref{dfn:totalvariationonmetricspaces} in such a way that it converges to $u$ also in $L^p(\Omega,\nu)$.

 Given $v \in L^p(\Omega,\nu)$, denote
$$ v_M(x) = \threepartdef{M}{v(x) > M;}{v(x)}{v(x) \in [-M,M];}{-M}{v(x) < -M.}$$
Observe that $v_M \in L^\infty(\Omega,\nu)$ and $v_M \rightarrow v$ in $L^p(\Omega,\nu)$ as $M \rightarrow \infty$.  Arguing as in the proof of Lemma \ref{lem:lipschitzapproximation}, we get that  there exists a sequence $(w_{n_k})_{M_k}$ such that
$$ (w_{n_k})_{M_{k}} \rightarrow u \quad \mbox{ in } L^p(\Omega,\nu).$$
Moreover, this sequence also converges strictly in $BV(\Omega,d,\nu)$, since truncations do not increase the slope:
$$ |Du|_\nu(\Omega) \leq \liminf_{k \rightarrow \infty} \int_\Omega |\nabla (w_{n_k})_{M_k}| \, d\nu \leq \liminf_{k \rightarrow \infty} \int_\Omega |\nabla w_{n_k}| \, d\nu = |Du|_\nu(\Omega).$$
Hence, possibly replacing the sequence $w_n$ by $(w_{n_k})_{M_k}$, we may require that $w_n \in L^p(\Omega,\nu)$ and $w_n \rightarrow u$ strictly in $BV(\Omega,d,\nu)$ and in the norm convergence in $L^p(\Omega,\nu)$. Hence, up to the modification of $w_n$ described above, we may require that $u_n$ satisfies these properties as well.

(3) First, notice that since $v_n$ are defined as discrete convolutions of $u$, we have $\| v_n \|_{L^\infty(\Omega,\nu)} \leq \| u \|_{L^\infty(\Omega,\nu)}$, since $v_n$ is a sum of averages of $u$ multiplied by a partition of unity. Moreover, $v_n \rightarrow u$ in $L^p(\Omega,\nu)$ for every $p \in [1,\infty)$. Then, notice that if $w_n \rightarrow u$ is the sequence given by  \eqref{dfn:totalvariationonmetricspaces}, then $(w_n)$ is bounded in $L^\infty(\Omega,\nu)$ by $\| u \|_{L^\infty(\Omega,\nu)}$ and by the argument from the proof of point (2) it converges to $u$ in $L^p(\Omega,\nu)$ for every $p \in [1,\infty)$.

Hence, the sequence $u_n$ is bounded in $L^\infty(\Omega,\nu)$ and converges to $u$ in $L^p(\Omega,\nu)$ for every $p \in [1,\infty)$. Hence, it admits a weakly* convergent subsequence $u_{n_k}$. By the uniqueness of the weak* limit, we have $u_{n_k} \rightharpoonup u$ weakly* in $L^\infty(\Omega,\nu)$.
\end{proof}

\subsection{Regular domains and Gauss-Green formula}

In order to provide a generalised Anzellotti pairing which satisfies a Green's formula, we will first need the Green's formula for Lipschitz functions in metric measure spaces. Recently, it was proved in \cite{BCM} under the assumption that $\Omega$ satisfies the following regularity assumption. Here, denote $$\Omega_t = \{ x \in \Omega: \, \mathrm{dist}(x, \Omega^c) \geq t \}.$$

\begin{definition}\label{def:regulardomain}
An open set $\Omega \subset \mathbb{X}$ is a {\it regular domain} if it has finite perimeter and
$$|D\1_\Omega|_\nu(\mathbb{X}) = \limsup_{t \rightarrow 0} \frac{\nu(\Omega \setminus \Omega_t)}{t}. $$
\end{definition}

An important feature of regular domains is that the infimum in the definition of the perimeter is achieved with an explicitly given sequence. Given a regular domain $\Omega$, for $t > 0$ we set
\begin{equation}
\varphi_t^\Omega(x) = \threepartdef{0}{x \in \Omega^c;}{\frac{\mathrm{dist}(x,\Omega^c)}{t}}{x \in \Omega \backslash \Omega_t;}{1}{x \in \Omega_t.}
\end{equation}
We call $\varphi_t^\Omega$ the {\it defining sequence} of $\Omega$. It is clear that $\varphi^\Omega_t \in \mathrm{Lip}_{b}(\mathbb{X})$ for all $t > 0$. The most important property of the defining sequence is given in the following Proposition.

\begin{proposition}\label{prop:weakconvergence}
Suppose that $\Omega \subset \mathbb{X}$ is a regular domain and $\varphi_t^\Omega$ is its defining sequence. Then,
\begin{equation*}
\lim_{t \rightarrow 0} \int_{\mathbb{X}} |\nabla \varphi_t^\Omega| \, d\nu =  |D\1_\Omega|_\nu(\mathbb{X}).
\end{equation*}
Moreover, we have  $|\nabla \varphi_t^\Omega| \, d\nu \rightharpoonup |D\1_\Omega|_\nu$ weakly* in $\mathcal{M}(\mathbb{X})$.
\end{proposition}

\begin{proof}
This is a consequence of the proof of \cite[Proposition 4.11]{BCM}, in particular of the equation (4.12) in that paper.
\end{proof}

In particular, this result implies that the notions of perimeter and the inner perimeter coincide for regular domains.

\begin{proposition}
Suppose that $\Omega \subset \mathbb{X}$ is a regular domain. Then, for any Borel set $A \subset \mathbb{X}$, the perimeter and inner perimeter of $\Omega$ in $A$ coincide, i.e.
\begin{equation*}
P_+(\Omega, A) = |D\1_\Omega|_\nu(A).
\end{equation*}
In particular, the measures $|D\1_\Omega|_\nu$ and $P_+(\Omega,\cdot)$ coincide.
\end{proposition}

\begin{proof}
First, notice that by definition for any Borel set $A \subset \mathbb{X}$ we have $|D\1_\Omega|_\nu(A) \leq P_+(\Omega, A)$, because the infimum on the left hand side is taken over a larger set.

Since $\Omega$ is a regular domain, take its defining sequence $\varphi_{t}^\Omega$. Notice that it is admissible in the infimum in the definition of the inner perimeter and compute
$$  |D\1_\Omega|_\nu(\mathbb{X}) \leq P_+(\Omega, \mathbb{X}) \leq \int_\mathbb{X} |\nabla \varphi_{t}^\Omega| \, d\nu $$
and by passing to the limit with $t \rightarrow 0$ we see that the first inequality is in fact an equality, i.e. $ |D\1_\Omega|_\nu(\mathbb{X}) = P_+(\Omega, \mathbb{X})$. Now, suppose that for some Borel set $A \subset \mathbb{X}$ we have $ |D\1_\Omega|_\nu(A) \neq P_+(\Omega, A)$. Then
$$
|D\1_\Omega|_\nu(\mathbb{X}) =  |D\1_\Omega|_\nu (A) + |D\1_\Omega|_\nu(\mathbb{X} \setminus A) $$ $$ < P_+(\Omega, A) + P_+(\Omega, \mathbb{X} \backslash A) = P_+(\Omega,\mathbb{X}) = |D\1_\Omega|_\nu(\mathbb{X}),
$$
a contradiction. Hence, the measures  $|D\1_\Omega|_\nu$ and $P_+(\Omega,\cdot)$ coincide.
\end{proof}

The reason why we will require our domain to be regular in the course of the paper is  to ensure the validity of a Gauss-Green theorem. Such a result on regular domains in metric measure spaces was proved in \cite[Theorem 4.13]{BCM}; here, we adapt this reasoning to prove its variant required for the proof of Theorem \ref{thm:generalgreensformula}, which is a generalisation of the Gauss-Green formula to BV functions. The difference is that we require that  $X \in L^\infty(T\Omega)$ instead of  $X \in L^\infty(T\mathbb{X})$, instead assuming a bit more on the domain to be able to use the first order differential structure as discussed in Remark \ref{rem:extensiontoclosure}. Because of this, we also need to use the correct notion of the divergence. Note that since the objects in  $L^\infty(T\Omega)$ are (a priori) not defined locally and only via duality, it is not immediately clear how to extend them to  $L^\infty(T\mathbb{X})$; we will instead prove the result directly and instead at some point use extensions of Lipschitz functions.

\begin{theorem}\label{thm:bufacomimiranda}
Suppose that $\Omega \subset \mathbb{X}$ is a regular domain and that  $X \in L^\infty(T\Omega)$ with  $\mbox{div}_0(X) \in L^1(\Omega,\nu)$. Then, there exists a function $(X \cdot \nu_\Omega)^- \in L^\infty(\partial\Omega, |D\1_{\Omega}|_\nu)$ such that
\begin{equation}
\int_\Omega f \, \mbox{div}_0(X) \, d\nu + \int_\Omega df(X) \, d\nu = - \int_{\partial\Omega} f (X \cdot \nu_\Omega)^- \, d | D\1_{\Omega} |_\nu
\end{equation}
for every $f\in \mbox{Lip}(\overline{\Omega})$. Moreover, we have the following estimate:
\begin{equation}
\| (X \cdot \nu_\Omega)^- \|_{L^\infty(\partial\Omega, |D\1_{\Omega}|_\nu)} \leq \| \vert X \vert \|_{L^\infty(\Omega, \nu)}.
\end{equation}
\end{theorem}

The object $(X \cdot \nu_\Omega)^-$ constructed in the above theorem is called the {\it interior normal trace} of $X$ on $\partial\Omega$. This Theorem can be understood as a relationship between the two definitions of the divergence $\mbox{div}(X)$ and $\mbox{div}_0(X)$; for every Lipschitz function $f \in \mbox{Lip}(\overline{\Omega})$, we have
\begin{equation*}
\int_{\overline{\Omega}} f \, d(\mbox{div}(X)) = - \int_\Omega df(X) \, d\nu = \int_\Omega f \, \mbox{div}_0(X) \, d\nu + \int_{\partial\Omega} f (X \cdot \nu_\Omega)^- \, d | D\1_{\Omega} |_\nu,
\end{equation*}
so the divergence $\mbox{div}(X)$ agrees with $\mbox{div}_0(X)$ inside $\Omega$ and also has a boundary term which includes the interior normal trace; in particular, whenever $X \in \mathcal{D}_0^q(\Omega)$, it typically lies does not lie in $\mathcal{D}^q(\Omega)$, but it lies in $\mathcal{DM}(\overline{\Omega})$.  Also, when the metric measure space is $(\mathbb{R}^N,d_{Eucl}, \mathcal{L}^N)$, then because the abstract vector fields and differentials coincide with their standard counterparts defined in coordinates, the interior normal trace also does.

\begin{proof}
Let $\varphi_t^\Omega$ be the defining sequence of $\Omega$. Since $f \varphi_t^\Omega \in \mathrm{Lip}(\overline{\Omega})$ and $f \varphi_t^\Omega $ has zero trace, by equation \eqref{eq:leibnizformulav2}
\begin{equation*}
\int_\Omega f \, d\varphi_t^\Omega(X) \, d\nu = \int_\Omega d(f \varphi_t^\Omega)(X) \, d\nu - \int_\Omega \varphi_t^\Omega \, df(X) \, d\nu \qquad\qquad\qquad\qquad\qquad\qquad
\end{equation*}
\begin{equation}\label{eq:firstlineofproofofgreen}
\qquad\qquad\qquad\qquad\qquad \qquad\qquad = - \int_\Omega f \, \varphi_t^\Omega \, \mathrm{div}_0(X) \, d\nu - \int_\Omega \varphi_t^\Omega \, df(X) \, d\nu.
\end{equation}
Since $\varphi_t^\Omega \rightarrow \1_\Omega$ a.e., by the dominated convergence theorem we may pass to the limit in both sides of the above equation and get
\begin{equation}\label{eq:secondlineofproofofgreen}
\lim_{t \rightarrow 0} \int_\Omega f \, d\varphi_t^\Omega(X) \, d\nu = - \int_\Omega f \,  \mathrm{div}_0(X) \, d\nu - \int_\Omega df(X) \, d\nu.
\end{equation}
Now, for any $f \in \mathrm{Lip}(\overline{\Omega})$ let us define the distribution
\begin{equation*}
T_X(f) = \lim_{t \rightarrow 0} T_X^t(f), \quad \mbox{where} \quad T_X^t(f) = \int_\Omega f d\varphi_t^\Omega(X) \, d\nu.
\end{equation*}
By equation \eqref{eq:firstlineofproofofgreen}, we get that
\begin{equation*}
|T_X^t(f)| \leq \| \vert X \vert \|_{L^\infty(\Omega,\nu)} \int_\Omega |f| |\nabla \varphi_t^\Omega| \, d\nu = \| \vert X \vert \|_{L^\infty(\Omega,\nu)} \int_{\mathbb{X}} |\overline{f}| |\nabla \varphi_t^\Omega| \, d\nu,
\end{equation*}
where $\overline{f}$ is any Lipschitz extension of $f$ to the whole space $\mathbb{X}$. By Proposition \ref{prop:weakconvergence}, we get
\begin{equation*}
|T_X(f)| \leq \| \vert X \vert \|_{L^\infty(\Omega,\nu)} \int_{\mathbb{X}} |\overline{f}|  \,  d|D\1_\Omega|_\nu = \| \vert X \vert \|_{L^\infty(\Omega,\nu)} \int_{\partial\Omega} |f|  \, d|D\1_\Omega|_\nu
\end{equation*}
\begin{equation*}
= \| \vert X \vert \|_{L^\infty(\Omega,\nu)} \| f \|_{L^1(\partial\Omega, |D\1_\Omega|)}.
\end{equation*}
Since $\mathrm{Lip}(\partial\Omega)$ is dense in $L^1(\partial\Omega, |D\1_\Omega|)$, the functional $T_X$ may be extended to a continuous functional on $L^1(\partial\Omega, |D\1_\Omega|)$; by the Riesz representation theorem (note that under the assumptions of Theorem \ref{thm:traces} the space $X$ is proper, see \cite[Proposition 3.1]{BB}) there exists a function $(X \cdot \nu_\Omega)^- \in L^\infty(\partial\Omega, |D\1_\Omega|_\nu)$ with norm at most $\| \vert X \vert \|_{L^\infty(\Omega,\nu)}$ such that
\begin{equation*}
T_X(f) = \int_{\partial\Omega} f \, (X \cdot \nu_\Omega)^- \,  d|D\1_\Omega|_\nu.
\end{equation*}
By equation \eqref{eq:secondlineofproofofgreen} we get the result.
\end{proof}

In the course of the paper, we will often require that the domain is regular and that it satisfies the assumptions of Theorem \ref{thm:traces}. Both types of conditions are regularity assumptions on the domain, and even though there are some similarities between the two conditions, the relationship between them is not straightforward. In short, regular domains may have interior or exterior cusps, but they may not have interior slits. This is discussed in the following Example.

\begin{example}\label{ex:regularnottraces}{\rm
Let $(\mathbb{R}^2,d,\mathcal{L}^2)$ be the standard two-dimensional Euclidean space. Consider the following three examples: \vspace{1mm} \\
(1) Let $\Omega \subset \mathbb{R}^2$ be set constructed as follows: we take a ball and add a single exterior cusp with the tip $x_0$, so that the resulting set is $C^\infty$ away from the tip. Then,
\begin{equation*}
\frac{\mathcal{L}^2(\Omega \setminus \Omega_t)}{t} = \frac{\mathcal{L}^2(\Omega \setminus (\Omega_t \cup B(x_0,2t)))}{t} + \frac{\mathcal{L}^2((\Omega \cap B(x_0,2t)) \setminus \Omega_t)}{t}.
\end{equation*}
In the first summand, since we are away from the tip, we have
\begin{equation*}
\lim_{t \rightarrow 0} \frac{\mathcal{L}^2(\Omega \setminus (\Omega_t \cup B(x_0,2t)))}{t} = \mathcal{H}^1(\partial\Omega).
\end{equation*}
In the second summand, we have
\begin{equation*}
\frac{\mathcal{L}^2((\Omega \cap B(x_0,2t)) \setminus \Omega_t)}{t} \leq \frac{\mathcal{L}^2(B(x_0,2t))}{t} \leq 4\pi t,
\end{equation*}
so it goes to zero as $t \rightarrow 0$. Hence, $\Omega$ is a regular domain. However, the measure density condition fails in the neighbourhood of $x_0$. On the other hand, the Ahlfors codimension one condition is satisfied and the weak $(1,1)-$Poincar\'e inequality holds. \vspace{1mm} \\
 (2) Let $\Omega \subset \mathbb{R}^2$ be the set considered in \cite{MSS}, i.e.
\begin{equation*}
\Omega = [-2,2]^2 \setminus \{ (x,y): \, -1 \leq x \leq 1, \, |x| \leq |y| \leq 1 \}.
\end{equation*}
Then, it is clear that $\Omega$ is a regular domain (it has Lipschitz boundary except for a single point). However, it does not satisfy the weak $(1,1)$-Poincar\'e inequality, as one can see by taking balls with small radii $r > 0$ and center $(-\frac{r}{2},0)$. On the other hand, $\Omega$ satisfies the measure density condition and the Ahlfors codimension one condition. \vspace{1mm} \\
(3) Let $\Omega \subset \mathbb{R}^2$ be the slit disk, i.e.
\begin{equation*}
\Omega = B(0,1) \setminus \{ (x,y): \, x \geq 0, \, y = 0 \}.
\end{equation*}
Then, $\Omega$ satisfies the measure density condition and the Ahlfors codimension one condition. However, it does not support a weak $(1,1)$-Poincar\'e inequality, as one can see by taking balls with small radii $r > 0$ and center $(\frac{1}{2},\frac{r}{2})$. Also, it is not a regular domain: notice that
\begin{equation*}
\lim_{r \rightarrow 0} \frac{\mathcal{L}^2(\Omega \backslash \Omega_t)}{t} = \mathcal{H}^1(\partial B(0,1)) + 2 \mathcal{H}^1(\{ (x,0): x \geq 0 \}) = 2\pi + 2,
\end{equation*}
but since $\Omega$ is the unit ball up to a set of zero Lebesgue measure, its perimeter measure equals $\mathcal{H}^1(\partial B (0,1)) = 2\pi$.
$\blacksquare$}
\end{example}

\subsection{Introducing the pairing}
	
In order to characterise the solutions to the least gradient problem, we will work with a metric analogue of the Anzellotti pairings between a vector field with integrable divergence and a BV function introduced in \cite{Anz} (see also \cite{KT}). Such an object, defined in terms of the differential structure due to Gigli, was constructed on the whole space $(\mathbb{X},d,\nu)$ in \cite{GM2021}. The most important property of this pairing is that it satisfies a Gauss-Green formula. The main goal of this subsection is to adapt this definition to the case when $\Omega \subset \mathbb{X}$ is an open bounded sufficiently regular set, and prove that the constructed pairing satisfies a version of the Gauss-Green formula which takes into account the boundary effects.

Suppose that $\Omega \subset \mathbb{X}$ satisfies the assumptions from the beginning of this Section. Assume that $X \in L^\infty(T\Omega)$ and $u \in BV(\Omega, d, \nu)$. As in the case of classical Anzellotti pairings, we will additionally assume  a joint regularity condition on $u$ and $X$ which makes the pairing well-defined. The condition is as follows: for $p \in [1,\infty)$, we have
\begin{equation}\label{Anzellotti:assumption}
 \mbox{div}_0(X) \in L^p(\Omega,\nu), \quad u \in BV(\Omega,d, \nu) \cap L^{q}(\Omega,\nu), \quad \frac{1}{p} + \frac{1}{q} = 1.
\end{equation}
 In other words, $X \in \mathcal{D}_0^{\infty,p}(\Omega)$. In the proofs, we will sometimes differentiate between the cases when $p > 1$ and $p = 1$.

\begin{definition}
Suppose that the pair $(X, u)$ satisfies the condition \eqref{Anzellotti:assumption}. Then, given a Lipschitz function  $f \in \mbox{Lip}(\Omega)$ with compact support, we set
$$ \langle (X, Du), f \rangle := -\int_\Omega u \,  \mbox{div}_0(fX) \, d\nu =  -\int_\Omega u \, df(X) \, d\nu - \int_\Omega u f \mbox{div}_0(X) \, d\nu.$$
\end{definition}

The following result was proved in \cite[Proposition 5.3]{GM2021}, with the proof being essentially the same on a bounded domain  with $\mbox{div}_0$ in place of $\mbox{div}$.

\begin{proposition}\label{prop:boundonAnzellottipairing}
$(X, Du)$ is a Radon measure which is absolutely continuous with respect to $|Du|_\nu$.  Moreover, for every Borel set $A \subset \Omega$ we have
$$ \int_A |(X,Du)| \leq \| \vert X \vert \|_\infty \int_A |Du|_\nu.$$
\end{proposition}

Before we prove the Gauss-Green formula, we require one more technical result,  see \cite[Lemma 5.4]{GM2021} (again, the proof on a bounded domain stays the same).
		
\begin{lemma}\label{lem:continuityofXDu}
Suppose that $u_i \rightarrow u$ as in the statement of Lemma \ref{lem:lipschitzapproximation}. Assume that the pair $(X,u)$ satisfies  the condition \eqref{Anzellotti:assumption}. Then
$$ \int_{\Omega} (X, Du_i) \rightarrow \int_\Omega (X,Du).$$
\end{lemma}

Now, we prove the main result of this Section, which is our main reason to consider the metric analogue of the Anzellotti pairings. Namely, we show that the Gauss-Green formula given in Theorem \ref{thm:bufacomimiranda} can be extended to the setting of BV functions in place of Lipschitz functions.
	
\begin{theorem}\label{thm:generalgreensformula}
Under the assumptions of Theorem \ref{thm:traces}, let $\Omega \subset \mathbb{X}$ be a regular domain and suppose that the pair $(X,u)$ satisfies  the condition \eqref{Anzellotti:assumption}. Then
$$ \int_\Omega u \,   \mbox{div}_0(X) \, d\nu  + \int_\Omega (X,Du) =  - \int_{\partial\Omega}  T_\Omega u \, (X \cdot \nu_\Omega)^- \, d| D_{\1_{\Omega}} |_\nu.$$
\end{theorem}

 Actually, a bit less is required - for the proof, we need only existence of traces of $BV$ functions and not existence of extensions. Hence, we do not need the Ahlfors codimension one condition from Theorem \ref{thm:traces}.

\begin{proof}
Since  $X \in L^\infty(T\Omega)$ is fixed and it has integrable divergence, the  normal trace $(X \cdot \nu_\Omega)^-$  is a fixed function from $L^\infty(\partial\Omega)$. We use Lemma \ref{lem:lipschitzapproximation}; given  $u \in BV(\Omega,d,\nu)$, we find a sequence $u_i \in Lip(\overline{\Omega})$ such that $u_i \rightarrow u$ strictly. We use Theorem \ref{thm:bufacomimiranda} with $f = u_i$ and obtain
\begin{equation}\label{eq:greenbeforelimit}
\int_\Omega u_i \, \mbox{div}_0(X) \, d\nu + \int_\Omega du_i(X) \, d\nu =  - \int_{\partial\Omega} u_i \, (X \cdot \nu_\Omega)^- \, d| D_{\1_{\Omega}} |_\nu.
\end{equation}
First, let us pass to the limit on the right hand side of equation \eqref{eq:greenbeforelimit}. Since the trace operator is continuous with respect to strict convergence, see for instance \cite[Proposition 7.1]{KLLS}, we have
\begin{equation}
\lim_{i \rightarrow \infty} \int_{\partial\Omega} u_i \,  (X \cdot \nu_\Omega)^- \, d| D_{\1_{\Omega}} |_\nu = \int_{\partial\Omega} T_\Omega u \,  (X \cdot \nu_\Omega)^- \, d| D_{\1_{\Omega}} |_\nu.
\end{equation}
As for the left hand side, first notice that because $u_i$ are Lipschitz, we have
$$\int_\Omega  du_i(X) \, d\nu = \int_\Omega (X, Du_i).$$
To see this, let $g \in \mbox{Lip}(\Omega)$ have bounded support in $\Omega$. Then, by the $L^\infty$-linearity of the differential, we have
$$ \int_\Omega g \, du_i(X) \, d\nu = \int_\Omega du_i(gX) \, d\nu = - \int_\Omega u_i \,  \mbox{div}_0(gX)  \, d\nu = \langle (X, Du_i), g \rangle.$$
Hence, integration with respect to $du_i(X) d\nu$ coincides with integration with respect $(X,Du_i)$ as a functional on Lipschitz functions with compact support, hence $du_i(X) d\nu$ and $(X,Du_i)$ coincide as measures.

Hence, on the left hand side, we have
$$ \lim_{i \rightarrow \infty} \bigg( \int_\Omega u_i \,  \mbox{div}_0(X) \, d\nu + \int_\Omega (X, Du_i) \bigg) =  \int_\Omega u \,  \mbox{div}_0(X)  \, d\nu + \int_\Omega (X, Du),$$
where we pass to the limit in the first summand using the assumption \eqref{Anzellotti:assumption} in the second summand using Lemma \ref{lem:continuityofXDu}. Hence, on both sides of equation \eqref{eq:greenbeforelimit} we may pass to the limit and obtain
$$ \int_\Omega u \, \mbox{div}_0(X) \, d\nu + \int_\Omega (X,Du) =  - \int_{\partial\Omega}   T_\Omega u \, (X \cdot \nu_\Omega)^- \, d| D_{\1_{\Omega}} |_\nu,$$
so the Theorem is proved.
\end{proof}

\section{Least gradient functions on metric measure spaces}\label{sec:leastgradient}
	
Now, we are ready to study least gradient functions on metric measure spaces.  In this Section, we also assume the structural assumptions on the metric measure space $(\mathbb{X},d,\nu)$ and the open bounded set $\Omega \subset \mathbb{X}$ introduced at the beginning of Section \ref{sec:Anzellotti}. Moreover, we frequently will additionally require that the assumptions that guarantee existence of traces are satisfied, such as in Theorem \ref{thm:traces}.

\begin{definition}\label{lg2}{\rm We say that $u\in BV(\Omega, d, \nu)$ is a {\it least gradient function} in $\Omega$, if
$$ \int_\Omega |Du|_\nu \leq  \int_\Omega |Dv|_\nu $$
for all $v \in BV(\Omega, d, \nu)$ such that  $T_\Omega u = T_\Omega v$}.
\end{definition}
	
\begin{remark}\label{equivalence} {\rm  In the Euclidean case the two definitions of functions of least gradient (Definition \ref{lg2} and equation \eqref{eq:leastgradient}) are equivalent (see \cite{SZ1}). It is clear that Definition \ref{lg2} implies equation \eqref{eq:leastgradient}, but the another implication is not immediate. However, it was proved in \cite[Proposition 9.3]{KLLS} that assuming $\mathcal{H}(\partial \Omega) < \infty$ both definitions are equivalent.$\blacksquare$}
\end{remark}

We want to give an equivalent characterisation using the newly defined pairing $(X, Du)$. Given $f \in L^1(\partial \Omega, \mathcal{H})$, we consider the energy functional $\mathcal{T}_f : L^1(\Omega, \nu) \rightarrow [0, + \infty]$ defined by the formula
\begin{equation}\label{fuctme}
\mathcal{T}_f (u):= \left\{ \begin{array}{ll} \vert Du \vert_\nu (\Omega) + \displaystyle\int_{\partial \Omega} \vert  T_\Omega(u) - f \vert \, d |D\1_\Omega|_\nu   \quad &\hbox{if} \ u \in BV(\Omega, d, \nu), \\ \\ + \infty \quad &\hbox{if} \ u \in  L^1(\Omega, \nu) \setminus BV(\Omega, d, \nu).\end{array}\right.
\end{equation}

In the proof of the main result in this Section, Theorem \ref{existme}, we will restrict the domain of definition of $\mathcal{T}_f$ to the Sobolev space $W^{1,1}(\Omega,d,\nu)$. Since $\nu(\partial\Omega) = 0$, given $u \in W^{1,1}(\Omega,d,\nu)$ we have
\begin{equation*}
\mathcal{T}_f (u)= \int_\Omega |Du|_\nu  + \displaystyle\int_{\partial \Omega} \vert  T_\Omega(u) - f \vert \, d |D\1_\Omega|_\nu = \| du \|_{L^1(T^* \Omega)}  + \displaystyle\int_{\partial \Omega} \vert  T_\Omega(u) - f \vert \, d |D\1_\Omega|_\nu.
\end{equation*}

\begin{proposition}\label{semicont}
Suppose that $\Omega$ and $\mathbb{X} \backslash \overline{\Omega}$ satisfy the assumptions of Theorem \ref{thm:traces}. Then, the functional $\mathcal{T}_f$ defined by the formula \eqref{fuctme} is convex and lower semicontinuous with respect to convergence in $L^1(\Omega,\nu)$.
\end{proposition}

\begin{proof}
It is clear that $\mathcal{T}_f$ is convex.  Let us see that $\mathcal{T}_f$ is lower semi-continuous with respect to the $L^1$-convergence. Without loss of generality, we can assume that $\mathbb{X} \setminus \overline{\Omega}$ is an open bounded set, since we can take   an open ball $B$ such that $\overline{\Omega} \subset B$, and work in the metric measure space $(B, d, \nu|_B)$.  By Theorem \ref{thm:traces}, there exists a function $\mbox{Ext}(f) \in BV(\mathbb{X} \backslash \overline{\Omega},d,\nu)$ such that $T_{\mathbb{X} \backslash \overline{\Omega}} (\mbox{Ext}(f)) = f$. Given $u \in BV(\Omega, d, \nu)$ and $f \in L^1(\partial \Omega, \mathcal{H})$, we denote $u_f:= u \1_\Omega + {\rm Ext}(f) \1_{\mathbb{X} \setminus \overline{\Omega}} \in L^1(\mathbb{X},\nu)$. Then, by \cite[Proposition 5.11]{HKLL} (see also  \cite[Proposition 7.5]{KLLS}) we have $u_f \in BV(\mathbb{X}, d, \nu)$ and we have
$$
\vert Du_f \vert_\nu (\mathbb{X}) = \vert Du_f \vert_\nu(\mathbb{X} \setminus  \partial_* \Omega)  +  \int_{\partial_* \Omega} \vert T_{\Omega} u_f - T_{\mathbb{X} \setminus \Omega} u_f \vert \,  d|D \1_\Omega|_\nu.
$$
Hence
\begin{equation}\label{lsc1}
\vert Du_f \vert_\nu (\mathbb{X}) = \vert Du \vert_\nu(\Omega)  + \vert D{\rm Ext}(f) \vert_\nu(\mathbb{X} \setminus \overline{ \Omega})  +  \int_{\partial_* \Omega} \vert T_{\Omega} u - f \vert \,   d|D\1_\Omega|_\nu.
\end{equation}
Given $u_n \in L^1(\Omega, \nu)$, such that $u_n \to u$ in $L^1(\Omega, \nu)$, we have $(u_n)_f \to u_f$ in $L^1(\mathbb{X}, \nu)$. Then, by the lower semi-continuity of the total variation, we have
$$\vert Du_f \vert_\nu (\mathbb{X}) \leq \liminf_{n \to \infty} \vert D(u_n)_f \vert_\nu (\mathbb{X}).$$
Therefore, by \eqref{lsc1}, we obtain that
$$\mathcal{T}_f (u) \leq \liminf_{n \to \infty} \mathcal{T}_f (u_n).$$
\end{proof}

\begin{proposition}\label{minimizer}
Suppose that $\Omega$ and $\mathbb{X} \backslash \overline{\Omega}$ satisfy the assumptions of Theorem \ref{thm:traces}. Then, the functional $\mathcal{T}_f$ defined by the formula \eqref{fuctme} has a minimiser.
\end{proposition}

\begin{proof}
 Let $B$ be a ball such that $\Omega \subset \subset B$ and $\nu(B \setminus \Omega) > 0$. Recall that (as in the proof of the previous Proposition) $\mathcal{T}_f(u)$ coincides with the total variation of some extension $u_f$; without loss of generality, we may assume that $u_f$ has support in the ball $B$. Then, by \cite[Lemma 2.2]{KKLS},
$$ \| u_f \|_{L^{Q/(Q-1)}(\mathbb{X},\nu)} \leq C |Du_f|_\nu(\mathbb{X}) = C_0 + C_1 \mathcal{T}_f(u),$$
where $Q$ is the homogenous dimension of $\mathbb{X}$ and $C$ depends on the radius of $B$, the doubling and Poincar\'e constants of $\mathbb{X}$ and the ratio of $\nu(B \setminus \Omega)/\nu(B)$. Hence, the functional $\mathcal{T}_f$ is coercive and minimising sequences are bounded in $L^{Q/(Q-1)}(\mathbb{X},\nu)$. Hence, we may apply the direct method, and get existence of a limiting function which lies in $L^{Q/(Q-1)}(\mathbb{X},\nu)$ which is a minimiser of the functional $\mathcal{T}_f$.
\end{proof}

Note that $u$ is a  solution to the Dirichlet problem of least gradient with boundary data $f$ in the sense of (T) if and only if $u$ is a minimiser of the energy functional $\mathcal{T}_f$, which is equivalent to
\begin{equation}\label{EL1}
0 \in \partial \mathcal{T}_f (u),
\end{equation}
that is the Euler-Lagrange of the variational problem, where $\partial \mathcal{T}_f$ is the subdifferential of the energy functional $ \mathcal{T}_f$. We formally write \eqref{EL1} as the following Dirichlet problem
\begin{equation}\label{DirichletME}
\left\{ \begin{array}{ll} - \Delta_{1,\nu} u = 0 \quad & \hbox{in} \ \ \Omega \\ \\ u = f  \quad & \hbox{on} \ \ \partial \Omega. \end{array} \right.
\end{equation}

Functions satisfying the PDE in \eqref{DirichletME} are called {\it 1-harmonic functions}. Following the characterisation of the subdifferential of $\Phi_h$ given in \cite{ABCM} (see also \cite{MazRoSe}), we give the following definition.
\begin{definition}\label{dfn:1laplace}{\rm
We say that $u$ is a solution of \eqref{DirichletME}, if there exists a vector field  $X \in \mathcal{D}_0^\infty(\Omega)$ with $\| X \|_\infty \leq 1$ such that the following conditions hold:
$$\mbox{div}_0(X) = 0 \quad \hbox{ in} \ \Omega; $$
$$ (X, Du) = |Du|_\nu \quad \hbox{as measures};$$
$$ (X \cdot \nu_\Omega)^- \in \mbox{sign}( T_\Omega u -f) \qquad  | D_{\1_{\Omega}} |_\nu-a.e. \ \hbox{on} \ \partial \Omega.$$}
\end{definition}

In order to prove the existence of solutions of problem \eqref{DirichletME} we need to use the version of the Fenchel-Rockafellar duality Theorem given in \cite[Remark 4.2]{EkelandTemam}.

Let $U,V$ be two Banach spaces and let $A: U \rightarrow V$ be a continuous linear operator. Denote by $A^*: V^* \rightarrow U^*$ its dual. Then, if the primal problem is of the form
\begin{equation}\tag{P}\label{eq:primal}
\inf_{u \in U} \bigg\{ E(Au) + G(u) \bigg\},
\end{equation}
then the dual problem is defined as the maximisation problem
\begin{equation}\tag{P*}\label{eq:dual}
\sup_{p^* \in V^*} \bigg\{ - E^*(-p^*) - G^*(A^* p^*) \bigg\},
\end{equation}
where $E^*$ and $G^*$ are the Legendre–Fenchel transformations (conjugate  functions) of $E$ and $G$ respectively, i.e.,
$$E^* (u^*):= \sup_{u \in U} \left\{ \langle u, u^* \rangle - E(u) \right\}.$$

\begin{theorem}[Fenchel-Rockafellar Duality Theorem]\label{FRTh} Assume that $E$ and $G$ are proper, convex and lower semi-continuous. If there exists $u_0 \in U$ such that $E(A u_0) < \infty$, $G(u_0) < \infty$ and $E$ is continuous at $A u_0$, then
$$\inf \eqref{eq:primal} = \sup \eqref{eq:dual}$$
and the dual problem \eqref{eq:dual} admits at least one solution. Moreover, the optimality condition of these two problems is given by
\begin{equation}\label{optimality}
A^* p^* \in \partial G(\overline{u}), \quad -p^* \in \partial E(A\overline{u})),
\end{equation}
where $\overline{u}$ is solution of \eqref{eq:primal} and $p^*$ is solution of \eqref{eq:dual}.
\end{theorem}

\begin{theorem}\label{existme}
Let $\Omega \subset \mathbb{X}$ be a regular domain such that $\Omega$ and $\mathbb{X} \backslash \overline{\Omega}$ satisfy the assumptions of Theorem \ref{thm:traces}. Then, for  each $f \in L^1(\partial \Omega, \mathcal{H})$ there exists a solution of \eqref{DirichletME}.
\end{theorem}

\begin{proof}
Let us express the minimisation of $\mathcal{T}_f$ in this framework. We restrict its domain of definition to $W^{1,1}(\Omega,d,\nu)$, so that the differential is a bounded operator from $W^{1,1}(\Omega,d,\nu)$ to $L^1(T^* \Omega)$. Therefore, we set  $U = W^{1,1}(\Omega,d,\nu)$, $V = L^1(\partial\Omega,|D\1_\Omega|_\nu) \times  L^1(T^{*} \Omega)$, and the operator $A: U \rightarrow V$ is defined by the formula
\begin{equation*}
Au = (T_\Omega u, du),
\end{equation*}
where $T_\Omega: BV(\Omega,d,\nu) \rightarrow L^1(\partial\Omega,\mathcal{H})$ is the trace operator in the sense of Definition \ref{dfn:trace} and $du$ is the differential of $u$ in the sense of Definition \ref{dfn:differential}. Hence, $A$ is a linear and continuous operator. Moreover, the dual spaces to $U$ and $V$ are
\begin{equation*}
U^* = (W^{1,1}(\Omega,d,\nu))^*, \qquad V^* =  L^\infty(\partial\Omega,|D\1_{\Omega}|_\nu) \times  L^\infty(T \Omega).
\end{equation*}
We denote the points $p \in V$ in the following way: $p = (p_0, \overline{p})$, where $p_0 \in L^1(\partial\Omega,\mathcal{H})$ and $\overline{p} \in  L^1(T^{*} \Omega)$. We will also use a similar notation for points $p^* \in V^*$. Then, we set $E: L^1(\partial\Omega,\mathcal{H}) \times  L^1(T^{*} \Omega) \rightarrow \mathbb{R}$ by the formula
\begin{equation}\label{eq:definitionofE}
E(p_0, \overline{p}) = E_0(p_0) + E_1(\overline{p}), \quad E_0(p_0) = \int_{\partial\Omega} |p_0 - f| \, d|D\1_{\Omega}|_\nu, \quad E_1(\overline{p}) = \| \overline{p} \|_{ L^1(T^{*} \Omega)}.
\end{equation}
We also set $G:  W^{1,1}(\Omega,d,\nu) \rightarrow \mathbb{R}$ to be the zero functional, i.e. $G \equiv 0$. In particular, the functional $G^*: (W^{1,1}(\Omega,d,\nu))^* \rightarrow [0,\infty]$ is given by the formula
\begin{equation*}
G^*(u^*) = \twopartdef{0}{\mbox{if } u^* = 0;}{+\infty}{\mbox{if } u^* \neq 0.}
\end{equation*}
The functional $E_0^*: L^\infty(\partial\Omega,\mathcal{H}) \rightarrow \mathbb{R} \cup \{ \infty \}$ is given by the formula
$$E_0^*(p_0^*) = \left\{ \begin{array}{ll} \displaystyle\int_{\partial\Omega} f \, p_0^* \, d|D\1_{\Omega}|_\nu, \quad &\hbox{if} \ \ |p_0^*| \leq 1, \\[10pt] +\infty,\quad &\hbox{otherwise.}   \end{array}  \right. $$
In fact, we only need to observe that
$$E_0^*(p_0^*) = \sup_{p_0 \in L^1(\partial\Omega,\mathcal{H})} \left\{ \int_{\partial \Omega} p_0 \, p_0^* \, d|D\1_{\Omega}|_\nu - \int_{\partial\Omega} |p_0 - f| \, d|D\1_{\Omega}|_\nu \right\}$$
$$= \sup_{p_0 \in L^1(\partial\Omega,\mathcal{H})} \left\{ \int_{\partial \Omega} (p_0 - f) \, p_0^* \, d|D\1_{\Omega}|_\nu + \int_{\partial\Omega} f \, p_0^* \, d|D\1_{\Omega}|_\nu  - \int_{\partial\Omega} |p_0 - f| \, d|D\1_{\Omega}|_\nu \right\}.$$
Let us see now that the functional $E_1^*:  L^\infty( T \Omega) \rightarrow [0,\infty]$ is given by the formula
$$E_1^*(\overline{p}^*) = \left\{ \begin{array}{ll}0, \quad &  \| \overline{p}^* \|_\infty \leq 1; \\[10pt] +\infty,\quad &\hbox{otherwise},\end{array}  \right. $$
i.e. $E_1^* = I_{B_1^\infty}$, the indicator function of the unit ball of  $L^\infty(T \Omega)$. In fact, we have $$(I_{B_1^\infty})^*(\overline{p}) =  \Vert \overline{p} \Vert_{L^1(T^* \Omega)},$$
and since $I_{B_1^\infty}$ is convex and lower semi-continuous, we have
$$I_{B_1^\infty} = (I_{B_1^\infty})^{**} = E_1^*. $$

The last thing we need to do is to check when $A^* p^* = 0$,  because the operator $A^*$ only enters the dual problem via $G^*(A^*p^*)$. By definition of the dual operator, for every  $u \in W^{1,1}(\Omega,d,\nu)$ we have
$$ 0 = \langle u, A^* p^* \rangle  = \langle p^*, Au \rangle = \int_{\partial\Omega} p^*_0 \, T_\Omega u \, d|D\1_{\Omega}|_\nu + \int_\Omega du(\overline{p}^*) \, d\nu.$$
First, take $u \in W_0^{1,1}(\Omega,d,\nu)$; then, this condition reduces to
$$ \int_\Omega du(\overline{p}^*) \, d\nu = 0,$$
so by the definition of the divergence we have $\mbox{div}_0(\overline{p}^*) = 0$. In particular, $X \in \mathcal{D}_0^\infty(\Omega)$.

%

In particular, for any $u \in  W^{1,1}(\Omega,d,\nu)$ we may use the Gauss-Green formula (Theorem \ref{thm:generalgreensformula}) to get
$$ 0 = \langle u, A^* p^* \rangle  = \langle p^*, Au \rangle = \int_{\partial\Omega} p^*_0 \, T_\Omega u \, d|D\1_{\Omega}|_\nu + \int_\Omega du(\overline{p}^*) \, d\nu = \int_{\partial\Omega} p^*_0 \, T_\Omega u \, d|D\1_{\Omega}|_\nu - $$
$$ - \int_\Omega u \,  \mbox{div}_0(\overline{p}^*)  \, d\nu - \int_{\partial\Omega} T_\Omega u \, (\overline{p}^* \cdot \nu_\Omega)^- \, d| D_{\1_{\Omega}} |_\nu = \int_{\partial\Omega} T_\Omega u \, (p_0^* - (\overline{p}^* \cdot \nu_\Omega)^-) \, d| D_{\1_{\Omega}} |_\nu.$$
Hence, because the right hand side disappears for all $u \in W^{1,1}(\Omega,d,\nu)$ and the trace operator from $W^{1,1}(\Omega,d,\nu)$ to $L^1(\partial\Omega,\mathcal{H})$ is surjective, we have $p_0^* = (\overline{p}^* \cdot \nu_\Omega)^-$.

Now, we give the exact form of the dual problem. We first rewrite the dual problem as
\begin{equation}
\sup_{p^* \in L^\infty(\partial\Omega,\mathcal{H}) \times   L^\infty(T \Omega)} \bigg\{ - E_0^*(-p_0^*) - E_1^*(\overline{p}^*) - G^*(A^* p^*) \bigg\}.
\end{equation}
Keeping in mind the above calculations, we set $\mathcal{Z}$ to be the subset of $V^*$ such that the dual problem does not immediately return $-\infty$, namely
\begin{equation}
\mathcal{Z} = \bigg\{ p^* \in L^\infty(\partial\Omega,\mathcal{H}) \times  \mathcal{D}_0^\infty(\Omega): \, \mbox{div}_0(\overline{p}^*)  = 0; \, \| \overline{p}^* \|_\infty \leq 1; \, \| p_0^* \|_\infty \leq 1; \, p_0^* = (\overline{p}^* \cdot \nu_\Omega)^- \bigg\}.
\end{equation}
Hence, we may rewrite the dual problem as
\begin{equation}
\sup_{p^* \in \mathcal{Z}} \bigg\{ - E_0^*(-p_0^*) \bigg\},
\end{equation}
so finally the dual problem takes the form
\begin{equation}\label{eq:finalmetricdual}
\sup_{p^* \in \mathcal{Z}} \bigg\{ - \int_{\partial\Omega} f \, p_0^* \, d|D\1_{\Omega}|_\nu \bigg\}.
\end{equation}
Moreover, for $u_0 \equiv 0$ we have  $E(Au_0) = \int_{\partial\Omega} |f| \, d|D\1_{\Omega}|_\nu < \infty$, $G(u_0) = 0 < \infty$ and $E$ is continuous at $0$. Then, by the Fenchel-Rockafellar Duality Theorem, we have
\begin{equation}\label{DFR1}\inf \eqref{eq:primal} = \sup \eqref{eq:dual}
\end{equation}
and
\begin{equation}\label{DFR2} \hbox{the dual problem \eqref{eq:dual} admits at least one solution.}
\end{equation}
In light of the constraint $p_0^* = (\overline{p}^* \cdot \nu_\Omega)^-$, we may simplify the dual problem a bit. We set
\begin{equation}
\mathcal{Z}' = \bigg\{ \overline{p}^* \in   \mathcal{D}_0^\infty(\Omega): \,  \mbox{div}_0(\overline{p}^*)  = 0; \, \| \overline{p}^* \|_\infty \leq 1 \bigg\}
\end{equation}
(note that in light of $p_0^* = (\overline{p}^* \cdot \nu_\Omega)^-$, the constraint $\| p_0^* \|_\infty \leq 1$ is automatically satisfied) and then we may reduce the dual problem to
\begin{equation}\label{eq:reducedmetricdual}
\sup_{\overline{p}^* \in \mathcal{Z}'} \bigg\{ - \int_{\partial\Omega} f \, (\overline{p}^* \cdot \nu_\Omega)^- \, d|D\1_{\Omega}|_\nu \bigg\}.
\end{equation}
Now, we will prove that if $p^*$ is a solution of the dual problem, then $-\overline{p}^*$ satisfies the conditions in Definition \ref{dfn:1laplace}.

Because the functional $\mathcal{T}_f$ is lower semicontinuous, we may use the $\varepsilon-$subdifferentiability property of minimising sequences, see \cite[Proposition V.1.2]{EkelandTemam}: for any minimising sequence $u_n$ for \eqref{eq:primal} and a maximiser $p^*$ of \eqref{eq:dual}, we have
\begin{equation}\label{eq:epsilonsubdiff1}
0 \leq E(Au_n) + E^*(-p^*) - \langle -p^*, Au_n \rangle \leq \varepsilon_n
\end{equation}
\begin{equation}\label{eq:epsilonsubdiff2}
0 \leq G(u_n) + G^*(A^* p^*) - \langle u_n, A^* p^* \rangle \leq \varepsilon_n
\end{equation}
with $\varepsilon_n \rightarrow 0$. Now, let $u \in BV(\Omega,d,\nu)$ be a minimiser of $\mathcal{T}_f$, given by Proposition \ref{minimizer}. Let us take a sequence $u_n \in W^{1,1}(\Omega,d,\nu)$ which has the same trace as $u$ and converges strictly to $u$ (as in Lemma \ref{lem:goodapproximation}); then, it is a minimising sequence in \eqref{eq:primal}. Equation \eqref{eq:epsilonsubdiff2} is automatically satisfied and equation \eqref{eq:epsilonsubdiff1} gives
\begin{equation}
0 \leq \int_{\partial\Omega} \bigg(|T_\Omega u_n - f| + p_0^* (T_\Omega u_n - f)  \bigg) \,  d|D\1_{\Omega}|_\nu + \bigg( \| du_n \|_{L^1(T^{*}\Omega)} + \int_\Omega du_n(\overline{p}^*) \, d\nu \bigg) \leq \varepsilon_n.
\end{equation}
Because the trace of $u_n$ is fixed (and equal to the trace of $u$), the integral on $\partial\Omega$ does not change with $n$; hence, it has to equal zero. Keeping in mind that $p_0^* = (\overline{p}^* \cdot \nu_\Omega)^-$, we get
$$ (\overline{p}^* \cdot \nu_\Omega)^- \in \mbox{sign}(T_\Omega u - f)\quad \mathcal{H}\mbox{-a.e. on } \partial\Omega.$$
Since the integral on $\partial\Omega$ equals zero and $ \| du_n \|_{L^1(T^{*} \Omega)} = \int_{\Omega} |Du_n|_\nu$, in the integral on $\Omega$ we have
\begin{equation}\label{eq:subdifferentiability}
0 \leq \int_\Omega \bigg( |Du_n|_\nu + du_n(\overline{p}^*) \, d\nu \bigg) \leq \varepsilon_n.
\end{equation}
Finally, keeping in mind that  $\mbox{div}_0(\overline{p}^*) = 0$ and again using the fact that the trace of $u_n$ is fixed and equal to the trace of $u$, by Gauss-Green's formula we get
$$ \int_\Omega du_n(\overline{p}^*) \, d\nu = - \int_{\partial\Omega} (\overline{p}^* \cdot \nu_\Omega)^- \, T_\Omega u_n \, d|D\1_\Omega|_\nu = -\int_{\partial\Omega} (\overline{p}^* \cdot \nu_\Omega)^- \, T_\Omega u \, d|D\1_\Omega|_\nu = \int_\Omega (\overline{p}^*, Du). $$
Hence, equation \eqref{eq:subdifferentiability} takes the form
\begin{equation*}
0 \leq \int_\Omega |Du_n|_\nu - \int_\Omega (-\overline{p}^*, Du)   \leq \varepsilon_n
\end{equation*}
and since $u_n$ converges strictly to $u$, we get that
$$ \int_\Omega |Du|_\nu - \int_\Omega (-p^*, Du) = \lim_{n \rightarrow \infty} \bigg( \int_\Omega |Du_n|_\nu - \int_\Omega (-\overline{p}^*, Du) \bigg) = 0.$$
This together with Proposition \ref{prop:boundonAnzellottipairing} implies that
\begin{equation*}
(-\overline{p}^*, Du) = |Du|_\nu \quad \mbox{as measures in } \Omega,
\end{equation*}
so the pair $(u, -\overline{p}^*)$ satisfies all the conditions in Definition \ref{dfn:1laplace}.
\end{proof}

\begin{theorem}\label{good}
Let $\Omega \subset \mathbb{X}$ be a regular domain such that $\Omega$ and $\mathbb{X} \backslash \overline{\Omega}$ satisfy the assumptions of Theorem \ref{thm:traces}. Let $f \in L^1(\partial \Omega, | D_{\1_{\Omega}} |_\nu)$. For $u \in BV(\Omega,d, \nu)$, the following are equivalent: \\
(i) $u$ is solution of problem \eqref{DirichletME}; \\
(ii) $0 \in \partial \mathcal{T}_f (u)$.
\end{theorem}

\begin{proof} Condition (i) implies (ii): given $w \in BV(\Omega, d, \nu)$, we apply the Gauss-Green formula (Theorem \ref{thm:generalgreensformula}) to obtain
$$0 = \int_\Omega (w -u) \, \mbox{div}_0(X)  \, d\nu =- \int_\Omega (X,D(w-u))  - \int_{\partial\Omega} (T_\Omega w - T_\Omega u)  \, (X \cdot \nu_\Omega)^- \, d | D\1_{\Omega} |_\nu$$ $$= - \int_\Omega (X,Dw) + \vert Du \vert_\nu (\Omega)- \int_{\partial\Omega} (T_\Omega w -f)  \, (X \cdot \nu_\Omega)^- \, d | D\1_{\Omega} |_\nu + \int_{\partial\Omega} (T_\Omega u -f)  \, (X \cdot \nu_\Omega)^- \, d | D\1_{\Omega} |_\nu.$$
Then,
$$\mathcal{T}_f (u)= \vert Du \vert_\nu (\Omega) + \displaystyle\int_{\partial \Omega} \vert T_\Omega u - f \vert d |D\1_{\Omega}|_\nu =  \int_\Omega (X,Dw)+ \int_{\partial\Omega} (T_\Omega w - f)  \, (X \cdot \nu_\Omega)^- \, d | D\1_{\Omega} |_\nu$$ $$\leq \vert Dw \vert_\nu (\Omega) + \displaystyle\int_{\partial \Omega} \vert T_\Omega w - f \vert \, d|D\1_{\Omega}|_\nu = \mathcal{T}_f (w).$$
Therefore,  $0 \in \partial \mathcal{T}_f (u)$.
		
\noindent Condition (ii) implies (i): By Theorem \ref{existme}, there is a function $\overline{u} \in BV(\Omega,d,\nu)$ which is a solution of problem \eqref{DirichletME}. Then, there exists a vector field  $\overline{X} \in\mathcal{D}_0^\infty(\Omega)$ with $\| \overline{X} \|_\infty \leq 1$ such that the following conditions hold:
$$ -  \mbox{div}_0(\overline{X}) = 0 \quad \hbox{in} \ \Omega; $$
$$ (\overline{X}, D\overline{u}) = |D\overline{u}|_\nu \quad \mbox{ as measures.}$$
$$ (\overline{X} \cdot \nu_\Omega)^- \in \mbox{sign}(T_\Omega \overline{u}-f) \qquad  | D_{\1_{\Omega}} |_\nu-a.e. \mbox{ on } \partial\Omega.$$
Then, applying Gauss-Green formula we get
$$ \int_\Omega (\overline{X}, D(\overline{u}-u)) = \int_{\partial\Omega} (T_\Omega \overline{u} - T_\Omega u) \, (\overline{X} \cdot \nu_\Omega)^- \, d | D\1_{\Omega} |_\nu,$$
and since $0 \in \partial \mathcal{T}_f (u)$, we get
$$\mathcal{T}_f (u) \leq \mathcal{T}_f (\overline{u}) = \vert D\overline{u} \vert_\nu (\Omega) + \displaystyle\int_{\partial \Omega} \vert T_\Omega \overline{u} - f \vert \, d|D\1_{\Omega}|_\nu $$ $$=\int_\Omega (\overline{X}, Du) + \int_{\partial\Omega} (f - T_\Omega u) \, (\overline{X} \cdot \nu_\Omega)^- \, d | D\1_{\Omega} |_\nu.$$
Hence,
$$\vert Du \vert_\nu (\Omega) - \int_\Omega (\overline{X}, Du) + \int_{\partial\Omega}(\vert T_\Omega u - f \vert - (\overline{X} \cdot \nu_\Omega)^- (f - T_\Omega u))\, d | D\1_{\Omega} |_\nu \leq 0.$$
Since both integrands are nonnegative, we deduce that $\vert Du \vert_\nu = (\overline{X}, Du)$ as measures in $\Omega$ and $\vert T_\Omega u - f \vert = (f - T_\Omega u) \, (\overline{X} \cdot \nu_\Omega)^-$ $| D\1_{\Omega} |_\nu$-a.e. on $\partial \Omega$, so that $$(\overline{X} \cdot \nu_\Omega)^- \in \mbox{sign}(T_\Omega u - f) \qquad \vert D\1_{\Omega} \vert_\nu-a.e. \mbox{ on } \partial\Omega.$$
Since  $\mbox{div}_0(\overline{X}) = 0$ in $\Omega$, we conclude that $u$ is a solution of problem \eqref{DirichletME}.
\end{proof}

In particular, the second part of the proof implies that a vector field $\overline{X}$ which is associated to a single solution of the least gradient problem actually works for all the solutions. This property was first observed in the Euclidean case in \cite{MazRoSe} (using a similar argument). Actually, we can give a better description as a consequence of the proof of Theorem \ref{existme}: any solution to the dual problem \eqref{eq:reducedmetricdual} works for all the solutions of the primal problem. This was first observed in the Euclidean case in \cite{Moradifam} (although the duality theory was applied differently). This is formalised in the following Corollary.

\begin{corollary}\label{cor:structure}
The structure of solutions is determined by a single vector field in the following sense: if  $-X \in \mathcal{D}_0^\infty(\Omega)$ is a solution of the dual problem \eqref{eq:reducedmetricdual}, then $X$ satisfies Definition \ref{dfn:1laplace} for all minimisers $u \in BV(\Omega,d,\nu)$ of the functional $\mathcal{T}_f$.
\end{corollary}

We need the following result, which is also of independent interest as an extension theorem for regular domains.

\begin{lemma}\label{need}
Let $\Omega \subset \mathbb{X}$ be a regular domain which satisfies the assumptions of Theorem \ref{thm:traces}. Let $h \in L^1(\partial \Omega,  |D\1_{\Omega} |_\nu)$. Given $\delta > 0$,  there exists $w \in BV(\Omega, d, \nu)$ such that $T_\Omega(w) = h$ and $$\vert D w \vert_\nu(\Omega) \leq \int_{\partial \Omega} \vert h \vert \, d | D\1_{\Omega} |_\nu + \delta.$$
\end{lemma}

 This result is an improvement of \cite[Proposition 4.1]{MSS}, where the inequality in the statement of Lemma \ref{need} is proved with a multiplicative constant on the right hand side. We prove that on regular domains we may take the constant to be equal to one.

\begin{proof} Given $h \in L^1(\partial \Omega,  |D\1_{\Omega} |_\nu)$, by Theorem \ref{thm:traces} there exists an extension $v := {\rm Ext}(h) \in BV(\Omega, d, \nu)$ with $T_\Omega v =h$. By Lemma \ref{lem:lipschitzapproximation}, there exists a sequence $v_n \in \mbox{Lip}(\overline{\Omega}) \cap BV(\Omega,d,\nu)$ such that $v_n \rightarrow v$ strictly in $BV(\Omega,d,\nu)$. We will also denote by $v_n$ its Lipschitz extension to $\mathbb{X}$.

Since the domain $\Omega$ is regular, let us take its defining sequence, i.e.
$$ \varphi_\varepsilon = \threepartdef{0}{x \in \Omega_\varepsilon}{\frac1\varepsilon \mbox{dist}(x, \Omega_\varepsilon)}{x \in \Omega \backslash \Omega_\varepsilon}{1}{x \in \Omega^c.}$$
By \cite[Proposition 4.11]{BCM} we have that $|\nabla \varphi_\varepsilon| \rightharpoonup |D\1_\Omega|$. Let $v_n^{\varepsilon}:= v_n \varphi_\varepsilon \in \mbox{Lip}(\overline{\Omega})$. Then,
$$ \int_\Omega |v_n| |\nabla \varphi_\varepsilon| \, d\nu \leq \int_{\mathbb{X}} |v_n| |\nabla \varphi_\varepsilon| \, d\nu \rightarrow \int_\mathbb{X} |v_n| \,d |D\1_{\Omega}|_\nu = \int_{\partial\Omega} |T_\Omega v_n| \,  d|D\1_\Omega|_\nu.$$
Moreover, we have that
$$ \int_\Omega |\nabla v_n| |\varphi_\varepsilon| \, d\nu \rightarrow 0,$$
since the measure of the support of $\varphi_\varepsilon$ goes to zero (by the regularity of the domain). Hence,
$$\lim_{\varepsilon \to 0} \vert Dv_n^{\varepsilon} \vert_\nu(\Omega) = \lim_{\varepsilon \to 0} \int_\Omega |\nabla(u \varphi_\varepsilon)| \, d\nu \leq \int_{\partial\Omega} |T_\Omega v_n| \, d|D\1_\Omega|_\nu.$$
Denote $h_n = h - T_\Omega v_n$. Since $v_n \rightarrow v$ strictly in $BV(\Omega,d,\nu)$, we have $h_n \rightarrow 0$ in $L^1(\partial\Omega,|D\1_\Omega|_\nu)$. Hence, we may rewrite the above inequality as
\begin{equation}\label{eq:estimateonvnepsilon}
\lim_{\varepsilon \to 0} \vert Dv_n^{\varepsilon} \vert_\nu(\Omega) \leq \int_{\partial\Omega} |h| \, d|D\1_\Omega|_\nu + \int_{\partial\Omega} |h_n| \,  d|D\1_\Omega|_\nu.
\end{equation}
Now, by \cite[Proposition 4.1]{MSS} there exists $w_n \in \mbox{Lip}_{loc}(\Omega) \cap BV(\Omega,d,\nu)$ with $T_\Omega w_n = h_n$ such that
\begin{equation}\label{eq:estimateonwn}
|Dw_n|_\nu(\Omega) \leq C \int_{\partial\Omega} |h_n| \, d|D\1_\Omega|_\nu.
\end{equation}
Finally, let us write $u_n^\varepsilon = v_n^\varepsilon + w_n \in \mbox{Lip}_{loc}(\Omega) \cap BV(\Omega,d,\nu)$. Then, $T_\Omega u_n^\varepsilon = h$. Moreover, since $h_n \rightarrow 0$ in $L^1(\partial\Omega, |D\1_\Omega|_\nu)$ and
\begin{equation*}
|Du_n^\varepsilon|_\nu(\Omega) \leq |Dv_n^\varepsilon|_\nu(\Omega) + |Dw_n|_\nu(\Omega),
\end{equation*}
inequalities \eqref{eq:estimateonvnepsilon} and \eqref{eq:estimateonwn} imply that given $\delta > 0$, for sufficiently small $\varepsilon$ and sufficiently large $n$ we have
$$\vert D u_n^\varepsilon \vert_\nu(\Omega) \leq  \int_{\partial \Omega} \vert h \vert \, d | D\1_{\Omega} |_\nu + \delta.$$
\end{proof}

\begin{corollary}\label{good1}
Let $\Omega \subset \mathbb{X}$ be a regular domain such that $\Omega$ and $\mathbb{X} \backslash \overline{\Omega}$ satisfy the assumptions of Theorem \ref{thm:traces}. Let $f \in L^1(\partial \Omega, | D\1_{\Omega} |_\nu)$. For $u \in BV(\Omega,d,\nu)$, satisfying $T_\Omega(u) = f$ on $\partial \Omega$, the following are equivalent:
\begin{itemize}
\item[(i)] $u$ is solution of problem \eqref{DirichletME},
\item[(ii)] $u$ is a minimiser of $\mathcal{T}_f$,
\item[(iii)] $u$ is a function of least gradient in $\Omega$.
\end{itemize}
\end{corollary}

\begin{proof}
By Theorem \ref{good}, we know that (i) and (ii) are equivalent.
		
Condition (ii) implies (iii): given $v \in BV(\Omega, d, \nu)$ such that $T_\Omega v = T_\Omega u = f$, we have
$$ \int_\Omega |Du|_\nu = \mathcal{T}_f(u) \leq  \mathcal{T}_f(v) =\int_\Omega |Dv|_\nu.$$
		
Condition (iii) implies (ii): fixed $v \in BV(\Omega, d, \nu)$, we need to show that $\mathcal{T}_f(u)  \leq  \mathcal{T}_f(v)$. By Lemma \ref{need},  given $\varepsilon > 0$,  there exists $w \in BV(\Omega, d, \nu)$ such that $T_\Omega(w) = T_\Omega(v) - f$ and $$\vert D w \vert_\nu(\Omega) \leq \int_{\partial \Omega} \vert T_\Omega(v) - f \vert \, d| D\1_{\Omega} |_\nu + \varepsilon.$$
Now, since $T_\Omega(v - w) = f = T_\Omega(u)$, by (iii), we have
$$\vert Du \vert_\nu (\Omega) \leq \vert D(v-w) \vert_\nu (\Omega) \leq \vert Dv \vert_\nu (\Omega)+ \vert Dw \vert_\nu (\Omega) $$
$$  \leq\vert Dv \vert_\nu (\Omega)+ \int_{\partial \Omega} \vert T_\Omega(v) - f \vert \,  d| D\1_{\Omega} |_\nu + \varepsilon.$$
Therefore,
$$\mathcal{T}_f(u) = \vert Du \vert_\nu (\Omega) \leq \mathcal{T}_f(v) + \varepsilon,$$
and since $\varepsilon >0$ is arbitrary, we obtain that $\mathcal{T}_f(u)  \leq  \mathcal{T}_f(v)$.
\end{proof}

Finally, let us comment on the connection between least gradient functions and sets with area-minimising boundary. In the Euclidean case, the first part, namely that boundaries of superlevel sets of least gradient functions are area-minimising, was established in \cite[Theorem 1]{BdGG}. The second part is that the converse holds true and was proved in \cite{SWZ}. To study the metric case, we will use the following definition.

\begin{definition}{\rm Let $E \subset \mathbb{X}$ be a set of finite perimeter in $\Omega$. We say that $\partial E$ is {\it area-minimising} in $\Omega$ if $\1_E$ is a function of least gradient in $\Omega$.
}
\end{definition}

Note that in the particular case of $(\R^N, d_{Eucl}, \mathcal{L}^N)$, the above concept coincides with the classical concept of set with area-minimising boundary (see \cite{SWZ} or \cite{SZ1}).

 Using a similar argument as in \cite{BdGG}, its metric analogue was established (with a slightly different definition of least gradient functions) in \cite[Lemma 3.5]{HKLS}. To be exact, it was shown that if $u$ is function of least gradient in $\Omega$, then for each $\lambda \in \R$ the set $\partial E_\lambda(u)$ is area-minimising in $\Omega$, where  $E_\lambda(u):= \{ x \in \Omega \ : \ u(x) \geq \lambda \}$. We will now prove that also the other implication is in the metric setting, and that it requires that this property holds only for almost all $\lambda \in \mathbb{R}$. Hence, whenever the superlevel sets are area-minimising for almost all $\lambda$, they are minimal for every $\lambda$. We will formulate the result in both directions and also give the proof of the metric analogue of the result \cite{BdGG} for the sake of completeness.

\begin{theorem}\label{leastMinim}
Let $\Omega \subset \mathbb{X}$ be an open set. Then, $u \in BV(\Omega,d,\nu)$ is function of least gradient in $\Omega$ if and only if $\partial E_\lambda(u)$ is area-minimising in $ \Omega$ for almost all (equivalently: all) $\lambda \in \R$.
\end{theorem}

\begin{proof}
First, assume that $\partial E_\lambda(u)$ is area-minimising in $ \Omega$ for almost all $\lambda \in \R$. Let $v \in BV(\Omega, d, \nu)$ such that $T_\Omega u = T_\Omega v$.  By Definition \ref{dfn:trace}, for all except countably many $\lambda \in \mathbb{R}$ we have $T_\Omega(\1_{E_\lambda(u)}) = T_\Omega(\1_{E_\lambda(v)})$. For such $\lambda$, by our assumption we have
$${\rm Per}_{\nu}(E_\lambda(u), \Omega) \leq {\rm Per}_{\nu}(E_\lambda(v), \Omega).$$
Therefore, by the coarea formula, we get
$$ \int_\Omega |Du|_\nu \leq  \int_\Omega |Dv|_\nu,$$
hence $u$ is function of least gradient in $\Omega$.

Assume now that $u$ is function of least gradient in $\Omega$. Let $u_1:= \max \{0, u - \lambda \}$ and $u_2:= u - u_1 = \min \{\lambda, u \}$. By the coarea formula, we have
$$ \int_\Omega |Du|_\nu =  \int_\Omega |Du_1|_\nu  + \int_\Omega |Du_2|_\nu.$$
Take $v \in BV(\Omega, d, \nu)$ with compact support. Then, since $u$ is function of least gradient in $\Omega$, we have
$$ \int_\Omega |Du|_\nu =  \int_\Omega |D(u+v)|_\nu \leq  \int_\Omega |D(u_1+v)|_\nu +  \int_\Omega |Du_2|_\nu$$
and
$$ \int_\Omega |Du|_\nu =  \int_\Omega |D(u+v)|_\nu \leq  \int_\Omega |D(u_2+v)|_\nu +  \int_\Omega |Du_1|_\nu.$$
Now, since we are assuming that $\mathcal{H}(\partial \Omega) < \infty$, by Remark \ref{equivalence}, we have that $u_1$ and $u_2$ are functions of least gradient in $\Omega$. Therefore, for every $\varepsilon >0$, the function
$$u_{\varepsilon, \lambda}:= \frac{1}{\varepsilon} \min \{\varepsilon, \max \{0, u - \lambda\}\}$$
is a function of least gradient in $\Omega$.

 Notice that for all but countably many $\lambda \in \mathbb{R}$
\begin{equation}\label{e1AMN}
\nu(\{ x \in \Omega \ : \ u(x) = \lambda \}) = 0.
\end{equation}
 For such $\lambda$, we have
$$u_{\varepsilon, \lambda} \to \1_{E_{\lambda}(u)} \quad \hbox{in} \ \ L^1(\Omega, \nu), \ \ \hbox{as} \ \ \varepsilon \downarrow  0.$$
 Hence, by \cite[Proposition 3.1]{HKLS}, which is a metric version of Miranda's theorem (\cite{Mir}), $\1_{E_\lambda(u)}$ is also a function of least gradient in $\Omega$.

On the other hand, if $\nu(\{ x \in \Omega \ : \ u(x) = \lambda \}) >0$, there exists a sequence $\{ \lambda_n \}$ such that $\lambda_n < \lambda$, $\lambda_n \to \lambda$, $\nu(\{ x \in \Omega \ : \ u(x) = \lambda_n \}) =0$ and
 $$\1_{E_{\lambda_n}}(u) \to \1_{E_{\lambda}(u)} \quad \hbox{in} \ \ L^1(\Omega, \nu), \ \ \hbox{as} \ \ n \to +\infty.$$
 Now, by the previous result, each $\1_{E_{\lambda_n}}(u)$ is a function of least gradient in $\Omega$, and therefore, aplying again the metric version of Miranda's theorem, we have that $\1_{E_{\lambda}}(u)$ is a function of least gradient in $\Omega$, and
consequently $\partial E_\lambda(u)$ is area-minimising in $\Omega$.

\end{proof}

\noindent {\bf Acknowledgment.} The first author has been partially supported by the DFG-FWF project FR 4083/3-1/I4354, by the OeAD-WTZ project CZ 01/2021, and by the project 2017/27/N/ST1/02418 funded by the National Science Centre, Poland. The second author has been partially supported by the Spanish MCIU and FEDER, project PGC2018-094775-B-100.

\end{document}